\documentclass[12pt]{amsart}
\usepackage[left=1in,top=1in,right=1in,bottom=1in,letterpaper]{geometry}
\usepackage[all]{xy}
\usepackage{epsfig}
\usepackage{amsmath}
\usepackage{amssymb}
\usepackage{amscd}
\usepackage{enumitem}
\usepackage{bm}
\usepackage{bbm}

\usepackage{url}
\usepackage{verbatim} 
\usepackage{color}
\usepackage{epsfig}
\usepackage{stmaryrd}
\usepackage{amsthm}
\usepackage{pifont}
\usepackage{mathrsfs}
\usepackage{wasysym}
\usepackage{hyperref}
\usepackage{graphicx}

\title[Immersions and translation structures I]{Immersions and translation structures I:\\
The space of structures on the pointed disk}

\author[]{W. Patrick Hooper}
\thanks{Support was provided by N.S.F. Grant DMS-1101233 and a PSC-CUNY Award (funded by The Professional Staff Congress and The City University of New York).}
\address{
The City College of New York\\
New York, NY, USA 10031}
\email{whooper@ccny.cuny.edu}

\ifdefined\theorem
\else
\newtheorem{theorem}{Theorem}
\newtheorem{proposition}[theorem]{Proposition}
\newtheorem{lemma}[theorem]{Lemma}
\newtheorem{remark}[theorem]{Remark}
\newtheorem{corollary}[theorem]{Corollary}

\newtheorem{example}[theorem]{Example}

\theoremstyle{definition}
\newtheorem{definition}[theorem]{Definition}

\fi
\ifdefined\openquestion
\else
\newtheorem{openquestion}[theorem]{Open Question}
\fi

\ifdefined\question
\else
\newtheorem{question}[theorem]{Question}
\fi

\newlength{\savearraycolsep}
	{\setlength{\savearraycolsep}{\arraycolsep}%
	\setlength{\arraycolsep}{#1}%
	\begin{array}{#2}}%
	{\end{array}\setlength{\arraycolsep}{\savearraycolsep}}

\newcommand{\set}[2]{{\{#1~:~\textrm{#2}\}}}

\def\C{\mathbb{C}}%
\def\N{\mathbb{N}}%
\def\R{\mathbb{R}}%
\def\Z{\mathbb{Z}}%
\def\GL{\textit{GL}}

\def\PSL{\textit{PSL}}

\def\0{{\mathbf{0}}}
\def\1{{\mathbf{1}}}

\def\v{{\mathbf{v}}} %

\def\sD{{\mathcal{D}}}
\def\sE{{\mathcal{E}}}
\def\sF{{\mathcal{F}}}
\def\sG{{\mathcal{G}}}

\def\sI{{\mathcal{I}}}
\def\sJ{{\mathcal{J}}}
\def\sK{{\mathcal{K}}}
\def\sL{{\mathcal{L}}}
\def\sM{{\mathcal{M}}}

\def\sP{{\mathcal{P}}}
\def\sR{{\mathcal{R}}}
\def\sS{{\mathcal{S}}}
\def\sT{{\mathcal{T}}}

\def\sV{{\mathcal{V}}}

\makeatletter
\def\imod#1{\allowbreak\mkern10mu({\operator@font mod}\,\,#1)}
\makeatother

\def\and{{\quad \textrm{and} \quad}}

\newtheorem{notational_convention}[theorem]{Notational Convention}

\newcommand{\PLH}{{\mathrm{PLH}}}
\newcommand{\Homeo}{{\textrm{Homeo}}}
\newcommand{\PC}{{\textrm{PC}}}
\newcommand{\disk}{{\textrm{Disk}}}
\newcommand{\cdisk}{{\overline{\disk}}}

\newcommand{\id}{{\textit{id}}}

\newcommand{\tM}{{\tilde \sM}}
\newcommand{\tE}{{\tilde \sE}}

\newcommand{\tpi}{{\tilde \pi}}

\newif\ifdraft\drafttrue
\draftfalse

\definecolor{darkgreen}{rgb}{0,0.8,0}
\newcommand{\com}[1]{\ifdraft{\textcolor{darkgreen}{FIX: #1}}\else\ignorespaces\fi}
\newcommand{\note}[1]{\ifdraft{\textcolor{blue}{NOTE: #1}}\else\ignorespaces\fi}

\def\emb{\hookrightarrow}
\def\imm{\rightsquigarrow}
\def\dev{{\textit{Dev}}} 
\def\Dev{{\textit{Dev}}} 
\def\fuse{\curlyvee}

\def\Fuse{\bigcurlyvee}
\def\Core{\bigcurlywedge}

\begin{document}
\begin{abstract}
We define a moduli space of translation structures on the open topological disk 
with a basepoint and endow it with a locally-compact metrizable topology. We call
this the immersive topology, because it is defined using the
concept of immersions: continuous maps between subsets of translation
surfaces that respect the basepoints and the translation structures.
Immersions induce a partial ordering on the moduli space, and
we prove the ordering is nearly a complete lattice in the sense of order
theory: The space is only missing a minimal element. Subsequent articles will
uncover more structure and develop a topology on the space of all translation structures.
\end{abstract}
\maketitle

\section*{Introduction}
A {\em translation structure} on a surface $\Sigma$ is an atlas of charts to the plane so that the transition functions are translations. It is a classical observation that by varying the geometry of translation structures on a surface of genus $g$ the structures can naturally converge to a translation structure on a surface of lower genus. 
Recent interest in studying dynamical and geometric properties of translation structures on surfaces of infinite genus has led to the need to take limits of sequences of translation surfaces whose genus is growing or limits of sequences of surfaces of infinite type whose topological type changes in the limit. In this and subsequent articles in this series, we provide a rigorous foundation for formulating such topological statements by considering the space of all translation structures (which simultaneously includes translation structures on surfaces of all topological types). This series of articles will include at least three articles:
\begin{enumerate}
\item[I.] In the current article, we define the {\em immersive topology} on the space of translation structures on the open disk $\Delta$ with basepoint $x_0$. We show this topology is locally compact and metrizable.
\item[II.] In the second article \cite{Immersions2}, we prove that the immersive topology on structures on the disk makes natural geometric and dynamical maps between translation structures on the pointed disk continuous. We also provide a practical method of proving convergence in this topology. This second article is still a work progress, but many of the results can also be found in the earlier preprint \cite{HooperImmersions1}.
\item[III.] In the third article \cite{Immersions3}, we will topologize the space of all translation structures on all surfaces with basepoint. Using the observation that the disk is homeomorphic to the universal cover of any translation surface admitting a translation structure, we push the topology on structures on the disk down to a topology on the space of all structures. This article will be an updated version of the preprint \cite{HooperImmersions2}.
\end{enumerate}

\section*{Annotated Table of Contents}

We provide an annotated table of contents in order to describe the structure of this document and summarize the main results of the paper. Note that sections \ref{sect:planar} and \ref{sect:topology disk 1} contain important definitions and concepts and should be read before setting further into the paper. 
You can get a sense of what is done in each section through the annotations below. Sections \S \ref{sect:basic} and later are structured so that the most important results are stated in the introduction of the section, with proofs and less important results appearing later in the section.

\newcommand{\tocitem}[1]{%
\item[{\hyperref[#1]{\S \ref*{#1}:}}]%
\hyperref[#1]{\textsc{\nameref*{#1}}} \dotfill \autopageref{#1}\\%
}

\begin{itemize}
\tocitem{sect:motivation}
We motivate this work by reviewing the classical theory of translation surfaces and discussing more recent research on translation surfaces of infinite type.

\tocitem{sect:planar}
This section describes the main objects of interest for the paper at a set theoretic level. These are the set $\tM$ of isomorphism classes of translation structures on the pointed disk and the canonical (set-theoretic) disk bundle $\tE$ over $\tM$. The fiber
in $\tE$ over a point of $\tM$ is naturally a pointed topological disk equipped with a translation structure in the isomorphism class of the image in $\tM$. We call these fibers {\em planar surfaces}; they are equivalent to but conceptually easier to work with than isomorphism class of translation structures on the disk (see Proposition \ref{prop:unique planar surface}).

\tocitem{sect:topology disk 1}
An immersion between two planar surfaces is a map respecting basepoints which act as a translation in local coordinates. Whenever an immersion exists, it is unique, and we say a planar surface $P$ {\em immerses} in a planar surface $Q$. An {\em embedding} is an injective immersion. The notion of immersion and embedding yield partial orders on $\tM$. We use these ideas to define the {\em immersive topologies} on $\tM$ and $\tE$.

\tocitem{sect:basic} 
We develop a basic toolkit for working with immersions and embedding between planar surfaces. These tools will be used throughout the paper. We use these basic tools to prove that the immersive topologies on $\tM$ and $\tE$ are Hausdorff.

\tocitem{sect:fusion}
The partial order defined by existence of immersions is nearly a complete lattice (in the sense of order theory): for any non-empty collection of planar surfaces, there is a  minimal (in the sense of the partial order) planar surface in which every surface in the collection immerses. We call this minimal planar surface the {\em fusion} of the collection. 
By adding a minimal element $O$ to the space of planar surfaces, we obtain a complete lattice. That is for any non-empty collection in $\tM \cup \{O\}$ there is a maximal element of $\tM \cup \{O\}$ which immerses in everything in the collection. We call this the {\em core} of the collection. 

\tocitem{sect:rectangular unions}
We introduce a tool for working with the topology. Namely, we study subsets of planar surfaces which are finite unions of rectangles. We use this to develop further intuition into the topologies. We find new natural open sets in $\tM$,
and prove that the immersive topologies on $\tM$ and $\tE$ have a countable basis.

\tocitem{sect:sequences}
Since $\tM$ and $\tE$ are Hausdorff and second-countable, many topological statements can be proved by considering sequences. To this end, we give necessary and sufficient conditions for sequences in these spaces to converge. 

\tocitem{sect:continuity of immersions}
We explain that immersions and embeddings between two planar surfaces are jointly continuous in the choice of two planar surfaces constituting the domain and range of the immersion.

\tocitem{sect:limits and compactness}
We prove that the only way a sequence of planar surfaces can fail to have a convergent subsequence is if the sequence of radii of the maximal Euclidean ball about the basepoints in the surfaces tends to zero. We use this fact to show
that the immersive topologies on $\tM$ and $\tE$ are locally compact and metrizable. As an ingredient in the proof, we show sequences of planar surfaces which are increasing in the sense of immersions converge to the fusion of the sequence.

\item[{\hyperref[sect:mcmullen]{Appendix:}}]%
\hyperref[sect:mcmullen]{\textsc{\nameref*{sect:mcmullen}}} \dotfill \autopageref{sect:mcmullen}\\%
We contrast the immersive topology with McMullen's geometric topology.

\end{itemize}

\section{Connections and motivation}
\label{sect:motivation}

We have taken some pains to distinguish translation structures (as defined above), from the notion of translation surface. A {\em translation surface} can be thought of as a pair $(X,\omega)$ consisting of a Riemann surface $X$ with a non-zero holomorphic $1$-form $\omega$ on $X$. The $1$-form 
$\omega$ can be integrated to obtain charts to the plane,
which are canonical up to postcomposition by translation. These charts are local homeomorphisms away from the set $Z \subset X$ of zeros of $\omega$, where cone singularities with cone angles in 
$2 \pi \Z$ appear. A translation surface $(X,\omega)$ gives rise to a translation structure on $X \smallsetminus Z$. 

Our understanding of translation surfaces has developed extensively since the pioneering work of Masur, Rauzy, Veech and others in the late 1970s and early 1980s. The field has attracted researchers from diverse fields in mathematics including
Teichm\"uller theory, algebraic geometry, and dynamical systems. Indeed, the interplay between these subjects has driven great progress in the field since its inception, and the field continues to be vibrant today.

We briefly discuss some well-established ideas relating to the study of the space of translation surfaces of genus $g \geq 1$, and to the convergence of a sequence of translation surfaces.
For more detail see the survey articles \cite[\S 6]{Smillie00}, \cite[\S 2]{MT} and \cite[\S 3.3]{Zorich06}.
Consider the moduli space of all translation surfaces of fixed genus. The collection of surfaces in this space with a 
fixed number of cone singularities with fixed cone angles is called a {\em stratum} of translation surfaces.
There is a well-studied way to place local coordinates on a stratum given via period coordinates. 
These coordinates give the stratum the structure of an orbifold with a locally affine structure. Of course, the cone singularities of a sequence of translation surfaces in a stratum can collide in the limit. In this case, there can still be a limiting translation surface but it lies outside the stratum. We
can take such a limit by viewing the space of translation surfaces of genus $g$ as identified with the collection
of pairs $(X,\omega)$, where $X$ is a Riemann surface and $\omega$ is a holomorphic $1$-form. As
such the space of translation surface of genus $g$ has the structure of a vector bundle over the moduli space $\sM_g$
of Riemann surfaces of genus $g$. 
Sequences of translation surfaces of genus $g$ can still leave this space. For instance,  a separating subsurface such as a cylinder could collapse to a point under a sequence. In order to take limits of such sequences, one can consider the Deligne-Mumford compactification of $\sM_g$, and 
make appropriate considerations for corresponding degenerations of holomorphic $1$-forms. See \cite[Definition 4.7]{MT}.

This work is primarily motivated by growing interest in the geometry and dynamics of translation surfaces of infinite topological type. Many works share a common interest in topological aspects of the space of all translation surfaces, e.g. to take limits of a sequence or to consider a continuously varying family. Examples of papers in which these ideas appears include \cite{Chamanara04}, \cite{HHW10} and \cite{Higl1} for instance. This sequence of papers will provide a firm foundation for making such statements.

To provide context, it is worth noting that even elementary questions about a single translation structure on an infinite type surface can be difficult to answer. Bowman and Valdez define the singularities of a translation structure to be points in the metric completion and study their structure (and the structure of so called ``linear approaches'' to the singularities) \cite{BV13}. It is not yet clear how the structure of singularities relates to the topology and geometry of the surface, a topic also investigated in \cite{Randecker2014} and \cite{CRW14}.

It is becoming increasingly clear that many of the techniques in the subject of translation surfaces are applicable
to the study of infinite translation surfaces. For instance, there is widespread interest in infinite abelian branched covers
of translation surfaces. Example articles include \cite{DHL14}, \cite{FU14}, \cite{HLT11}, \cite{HS09}, \cite{HW13}, \cite{RTarxiv12}, \cite{RTarxiv11} and \cite{Schmoll11}. 
More relevantly, there is some work to suggest that many methods in use are applicable to infinite translation surfaces which do not arise from covering constructions. 
In the finite genus case, Masur's criterion \cite{M82} says that if the orbit of a translation surface under the Teichm\"uller flow recurs than the vertical straight line flow is uniquely ergodic. It would be nice to have such
a statement in the case where surfaces have infinite genus but finite area. There are has been partial progress. The article \cite{Hinf} concludes ergodic theoretic results such as unique ergodicity about certain special surfaces from a notion of recurrence in the spirit of Masur's criterion. In \cite{Trev14}, a criterion was described for ergodicity of the straight-line flow in translation surfaces of infinite topological type but finite area in terms of a rate of degeneration of a certain geometric quantity under the Teichm\"uller flow (showing that a topology on moduli space may not be necessary for ergodic theoretic results).

Problems in Teichm\"uller theory and 3-manifold topology prompted McMullen to place a geometric topology on the space of all Riemann surfaces with base frames paired with a quadratic differential; see \cite[Appendix]{McMullenAmenability} and \cite[\S 2.3]{McMullenIteration}. Our approach has slightly different aims and results in a slightly different topology; see the discussion in \hyperref[sect:mcmullen]{the appendix}.

Finally, the unfolding construction in polygonal billiards (due to \cite{FK36} and \cite{ZK}) works in the case where the polygon's angles are irrational multiples of $\pi$. We intend to use the topology to prove Conjecture 1.6 of \cite{HooSch09}. The main method of proof involves stretching these surfaces under a divergent sequence of maps in $\GL(2,\R)$ and passing to limit surfaces.

\section{The set of translation structures on the disk}
\label{sect:planar}

\subsection{Translation structures}
Let $\Sigma$ denote an oriented topological surface. A translation structure on $\Sigma$ is a $(G,X)$-structure on $\Sigma$ in the sense of Thurston (see e.g., \cite{Thurston}) where $G$ is the group of translations acting on $X=\R^2$, 
i.e., an atlas of charts $\{(U_j,\phi_j)~:~j \in \sJ\}$ to the plane so that the transition maps are locally restrictions of  translations. Here, we insist that each chart $\phi_j:U_j \to \R^2$ be orientation preserving.

Two translation structures on a oriented surface $\Sigma$ are the {\em same} if the union of the two atlases still determines a translation structure. We say the translation structures $\{(U_j,\phi_j):j \in \sJ\}$ on $\Sigma_1$
and $\{(V_k,\psi_k):k \in \sK\}$ on $\Sigma_2$ are {\em translation isomorphic} if there is an orientation preserving homeomorphism $h:\Sigma_1 \to \Sigma_2$ so that the structure determined by
$\{\big(h(U_j),\phi_j\circ h^{-1}\big):j \in \sJ\}$
is the same as the structure determined by $\{(V_k,\psi_k):k \in \sK\}$. The homeomorphism 
$h$ is called a {\em translation isomorphism} from the first structure to the second. 

\begin{remark}
\label{remark:removing cone singularities}
We do not allow cone singularities in our translation structures. A translation structure may be obtained from a translation surface with cone singularities by removing the singular points from the underlying topological surface.
\end{remark}

\subsection{Translation structures on the pointed disk}

Let $\Delta_1$ and $\Delta_2$ be oriented open topological $2$-dimensional disks with basepoints $x_1$ and $x_2$, respectively.
We say translation structures on these two disks are {\em isomorphic} if there is a translation isomorphism from the first structure to the second which respects basepoints. We call a translation isomorphism which respects basepoints an {\em isomorphism}.

Note that a translation structure on $\Delta_1$ is always isomorphic to a translation structure on $\Delta_2$; simply push the structure forward under an orientation preserving homeomorphism $\Delta_1 \to \Delta_2$. 

\begin{remark}[Examples of structures on the disk]
We briefly state the main construction of interest to us. If a surface $\Sigma$ admits a translation structure
with a selected basepoint, then its universal cover is a topological disk with basepoint. We can then lift the translation structure to the disk. This can be done, for instance, for translation surfaces homeomorphic to a closed surface with the cone singularities removed; see Remark
\ref{remark:removing cone singularities}.
There are other translation structures on the pointed disk as well. Any open topological disk in the plane containing the origin $\0$ admits a natural translation structure. This also works more generally: an open topological disk containing the basepoint in a surface with a translation structure gives a translation structure on the disk.
\end{remark}

\subsection{The set-theoretic moduli space}
Throughout this paper, $\Delta$ denotes a fixed choice of an oriented open topological $2$-dimensional disk with basepoint $x_0 \in \Delta$. 
Let $\{(U_j,\phi_j): j \in \sJ\}$ be an
atlas of charts determining a translation structure on $\Delta$. Because $\Delta$ is simply connected, 
by analytic continuation there is a
unique map $\phi:\Delta \to \R^2$ so that
\begin{enumerate}
\item $\phi(x_0)=\0$, where $\0=(0,0) \in \R^2$. 
\item For each $j \in \sJ$, the map $\phi|_{U_j}$ agrees with $\phi_j$ up to postcomposition with a translation.
\end{enumerate} 
This map $\phi$ is a local homeomorphism called the {\em developing map} in the language of $(G,X)$ structures. See \cite[\S 3.4]{Thurston}. 
Observe that the single chart $(\Delta,\phi)$ determines a translation surface structure on $\Delta$ which is the same as the original structure. Conversely,
each orientation preserving local homeomorphism $\phi:\Delta \to \R^2$ determines a translation structure: the one determined by the atlas $\{(\Delta,\phi)\}$. 

We will say a {\em pointed local homeomorphism} (from $(\Delta,x_0)$ to $(\R^2,\0)$) is an
orientation preserving local homeomorphism $\phi:\Delta \to \R^2$ so that $\phi(x_0)=\0$. We use $\PLH$
to denote the collection of all such maps. Observe that our developing maps lie in $\PLH$, and $\PLH$ is naturally identified with the collection of all translation structures
on $(\Delta,x_0)$ modulo sameness.

Let $\Homeo_+(\Delta,x_0)$ denote the group of orientation preserving homeomorphisms $\Delta \to \Delta$ which fix the basepoint $x_0$. 
Note that the translation structures on $\Delta$ determined by 
$\phi$ and $\psi$ in $\PLH$ are isomorphic if and only if there is an
$h \in \Homeo_+(\Delta,x_0)$
so that $\psi=\phi \circ h^{-1}.$ Thus, the (set-theoretic) moduli space of all translation structures on $(\Delta,x_0)$ modulo isomorphism is given by 
$$\tM=\PLH/\Homeo_+(\Delta,x_0).$$
We use the notation $[\phi]$ to indicate the $\Homeo_+(\Delta,x_0)$-equivalence class of $\phi \in \PLH$.

\begin{remark}[Quotient topology]
\label{rem:quotient topology}
One can endow $PLH$ with the compact-open topology and $\tM$ with the resulting quotient topology. This is not what we do in this paper, because the resulting topology is not Hausdorff. The open unit disk and the plane can be considered to be points in $\tM$, and 
every open set containing the unit disk in the quotient topology also contains the plane.
\end{remark}

\subsection{The disk bundle over moduli space}
The group $\Homeo_+(\Delta,x_0)$ naturally acts on $\PLH \times \Delta$ by 
$$(\phi,y) \mapsto \big(\phi \circ h^{-1}, h(y)\big)
\quad \text{for $h \in \Homeo_+(\Delta,x_0)$.}$$
The {\em canonical disk bundle over $\tM$} is given by
$$\tE=\big(\PLH \times \Delta\big)/\Homeo_+(\Delta,x_0).$$
We denote the $\Homeo_+(\Delta,x_0)$-equivalence class of $(\phi, y)$ by $[\phi, y] \in \tE$. 

Because of the description of the $\Homeo_+(\Delta,x_0)$-action, there is a canonical map
\begin{equation}
\label{eq:dev}
\dev:\tE \to \R^2; \quad [\phi,y] \mapsto \phi(y).
\end{equation}
We call this map the {\em (bundle-wide) developing map}.
There is also a natural projection from $\tE$ on to the moduli space $\tM$ given by
$$\tpi:\tE \to \tM; \quad [\phi,y] \mapsto [\phi].$$

\subsection{Structures on the fibers}
\label{sect: fibers}
We will call each fiber $\tpi^{-1}([\phi])$ {\em a planar surface}.
Observe that the choice of a representative $\phi \in [\phi]$ yields an identification of the planar surface $\tpi^{-1}([\phi])$ with $\Delta$:
\begin{equation}
\label{eq:i phi}
i_\phi:\Delta \to \tpi^{-1}([\phi]); \quad y \mapsto [\phi,y].
\end{equation}
We endow the fiber $\tpi^{-1}([\phi])$ with the topology and orientation which make $i_\phi$ and orientation preserving homeomorphism and note that the selected orientation and topology is is independent of the choice of $\phi \in [\phi]$. The map $i_\phi$ can also be used to push the basepoint $x_0\in \Delta$ onto the planar surface. We treat $[\phi,x_0]$ as the basepoint of the planar surface, and note that this point is also independent of the choice of $\phi$. Finally, we note that the restriction of the developing map to the fiber is an orientation preserving local homeomorphism to the plane which sends the basepoint $[\phi,x_0]$ to $\0$. This endows the fiber $\tpi^{-1}([\phi])$
with the structure of a translation surface (using an atlas consisting of only the developing map restricted to the fiber). The translation structure on the fiber $\tpi^{-1}([\phi])$ with basepoint $[\phi, x_0]$ is isomorphic to those in the pointed translation structures in the equivalence class $[\phi]$. 

These ideas lead to the following result:
\begin{proposition}
\label{prop:unique planar surface}
Any translation structure on an oriented open $2$-dimensional topological disk with basepoint is isomorphic to a unique planar surface.
\end{proposition}

\begin{notational_convention}\label{conv}
We follow some conventions to simplify notation by effectively removing the need to discuss equivalence classes. Formally, the planar surfaces are parametrized by 
equivalence classes $[\phi] \in \tM$ via $P=\tpi^{-1}([\phi])$. We will identify the objects $P$ and $[\phi]$, and thus we can more simply write $P \in \tM$. We will typically denote the basepoint of $P$ by
$o_P$; as noted above $o_P=[\phi,x_0]$. We will denote points of $P$ by letters
such as $p,q \in P$. Note that points of $P$ are also points of $\tE$. But, we will rarely refer to a point $p \in \tE$ without referring 
to the planar surface $P=\tpi(p) \in \tM$ which contains $p$. Therefore, we will redundantly refer to points of $\tE$ as pairs $(P,p)$ where $P$ is a planar surface and $p \in P$ is a point in this surface (i.e., $P=\tpi(p)$).
\end{notational_convention}

\section{Immersions and a topology on the moduli space} 
\label{sect:topology disk 1}

\subsection{Definition of immersion}
\label{sect:def immersion}
Let $P$ be a planar surface, following the discussion in 
\S \ref{sect:planar} including Convention \ref{conv}.
We define $\PC(P)$ to be the collection of all path-connected subsets of $P$ which contain the basepoint $o_P \in P$. 

Let $P$ and $Q$ be planar surfaces, and choose $A \in \PC(P)$ and $B \in \PC(Q)$. We say ``$A$ {\em immerses} into $B$'' and write ``$A \imm B$'' if there is a continuous map
$\iota:A \to B$ which acts as a translation in local coordinates so that $\iota(o_P)=o_Q$. 
We call the map $\iota$ an {\em immersion (respecting the translation structures)}.
We will write
``$A \not \imm B$'' to indicate that $A$ does not immerse
in $B$, and will write
``$\exists \iota:A \imm B$'' as shorthand for the phase ``there exists an immersion $\iota$ from $A$ to $B$.'' 

\begin{example}
If one takes a finite genus translation surface with singularities removed, the developing map applied to the universal cover is an immersion of this cover into the plane.
\end{example}

An {\em embedding} is an injective immersion. If such a map exists between an 
$A \in \PC(P)$ and $B \in \PC(Q)$, we say ``$A$ embeds in $B$'' and write ``$A \emb B$.'' We follow notational
conventions as for immersions.

We make some basic observations about immersions:
\begin{proposition}[Uniqueness of immersions]
\label{prop:unique}
For $A \in \PC(P)$ and $B \in \PC(Q)$, there is at most one immersion from $A$ into $B$.
\end{proposition}
\begin{proof}
Suppose $\iota_1$ and $\iota_2$ are immersions from $A$ to $B$. We will show that $\iota_1(a)=\iota_2(a)$ for all $a \in A$. The set of $a$ satisfying this equation is open (because both $\iota_1$ and $\iota_2$ are local translations)
and closed (because of continuity of the maps). Also it holds at $o_P \in A$ by definition of immersion.
Since $A$ is path connected, it must hold in all of $A$. 
\end{proof}

An {\em isomorphism} between $A \in \PC(P)$ and $B \in \PC(Q)$ is an immersion of $A$ into $B$
which is also a homeomorphism. Compositions and inverses of isomorphisms are still isomorphisms.
This implies that the notion of isomorphism gives an equivalence relation on $\bigcup_{P \in \tM} \PC(P)$
and we use $\PC$ to denote the collection of such isomorphism classes.
Observe that if $A$ and $A'$ are isomorphic and $B$ and $B'$ are isomorphic then $A \imm B$
if and only if $A' \imm B'$. Therefore, the notion of immerses gives an relation
on the collection $\PC$ of isomorphism classes of elements of $\bigcup_{P \in \tM} \PC(P)$. 

\begin{remark}
We will abuse notation by typically ignoring the distinction between an $A \in \PC(P)$ and
its equivalence class in $\PC$. If $A \in \PC$, then a point in $A$ would formally be an
isomorphism class of pairs $(A,a)$ much as in our definition of $\tE$. We are not interested
in giving more structure to $\PC$, so we will ignore this issue. 
\end{remark}

\begin{corollary}[Partial ordering]
\label{cor:partial ordering}
The notions of immerses ($\imm$) and embeds ($\emb$) viewed as relations on $\PC$ are partial orderings.
Both notions restrict to partial orderings on $\tM$. 
\end{corollary}

\begin{proof}
Let $A,B,C \in \PC$. 
The identity map is an embedding $A \emb A$, so $\imm$ and $\emb$ are reflective.
If $A \imm B$ and $B \imm C$, then the composition of immersions gives an immersion $A \imm C$, so $\imm$ is transitive. Furthermore, if both maps are injective then so is the composition, so $\emb$ is transitive.
We must show that $A \imm B$ and $B \imm A$ implies $A$ and $B$ are isomorphic. 
Let $\iota_1:A \imm B$ and $\iota_2:B \imm A$. Then
$\iota_2 \circ \iota_1$ is an immersion of $A$ into itself.
Since the identity map on $A$ is also an immersion of $A$ into itself, by the uniqueness of immersions, $\iota_2 \circ \iota_1=\id_A$. Similarly, $\iota_1 \circ \iota_2=\id_B$. The immersions are inverses of one another, and hence are homeomorphisms. Furthermore, the fact that the maps respect the basepoints implies that they are isomorphisms. This also holds for embeddings since embeddings are a special case of immersions.
\end{proof}

\begin{proposition}
\label{prop:dev and immersions}
Let $P$ and $Q$ be planar surfaces and let $A \in \PC(P)$ and $B \in \PC(P)$. A continuous function $f:A \to B$
satisfying $f(o_P)=o_Q$ is an immersion if and only if $\dev(P,a)=\dev\big(Q,f(a)\big)$ for all $a \in A$.
\end{proposition}
As a consequence of this, we see that the developing map is well defined when restricted to an $A \in \PC$.
That is, if $A \in \PC(P)$ and $B \in \PC(Q)$ are isomorphic then the isomorphism does not affect the image under the developing map.
\begin{proof}
Clearly if $\dev(P,a)=\dev\big(Q,f(a)\big)$ holds for all $a \in A$, then $f$ is a local translation
and so is an immersion. Conversely suppose $f$ is an immersion. 
Note that the developing map of a planar surface $\dev|_P:P \to \R^2$ is an immersion if $\R^2$ is considered to be a planar surface with basepoint $\0$. Since compositions of immersions are immersions,
we know $\dev \circ f$ is an immersion of $A$ into $\R^2$. Therefore, this map agrees with the restriction of $\dev$ to $A$. 
\end{proof}

\subsection{The topology on moduli space}
\label{sect:topology}
We will specify the topology on the moduli space of all planar surfaces, $\tM$, by specifying a subbasis for the topology. That is, we will be concerned with the coarsest topology which makes
a collection of sets open. 

Let $P\in \tM$ be a planar surface and let $K \in \PC(P)$ be a compact subset of $P$. We define the following subsets of the moduli space $\tM$:
\begin{equation}
\label{eq:n+}
\tM_{\imm}(K)=\{Q \in \tM~:~K \imm Q\}
\and
\tM_{\emb}(K)=\{Q \in \tM~:~K \emb Q\}.
\end{equation}
Sets of this form will be open in our topology.
However, they are insufficient to form a subbasis for a Hausdorff topology, because they fail to isolate points. For instance, the plane (interpreted as a planar surface) lies in each set $\tM_{\imm}(K)$. Also, any set of the form $\tM_{\emb}(K)$ which contains the unit disk also contains the plane. 

Let $P$ be a planar surface, and let $U \in \PC(P)$ be open as a subset of $P$. We define:
\begin{equation}
\label{eq:n-}
\tM_{\not \imm}(U)=\{Q \in \tM~:~U \not \imm Q\} \and \tM_{\not \emb}(U)=\{Q \in \tM~:~U \not \emb Q\}.
\end{equation}
These sets will also be open in our topology.

We would like to describe a subbasis for our topology which
consists of sets which are fairly easy to work with. 
So, for any planar surface $P$, we will distinguish two natural subsets of $\PC(P)$.
We define $\cdisk(P)$ to be those sets in $\PC(P)$ which are homeomorphic to a closed
$2$-dimensional disk and contain the basepoint $o_P$ in their interior. We define $\disk(P)$ to be the set of sets in $\PC(P)$ which are homeomorphic to an open $2$-dimensional disk and contain the basepoint.

\begin{definition}
\label{def:immersive topology on M}
The {\em immersive topology} on $\tM$ is the coarsest topology so that sets of either of the two forms below are open:
\begin{enumerate}
\item Sets of the form $\tM_{\imm}(K)$ where $K \in \cdisk(P)$ for some $P \in \tM$. 
\item Sets of the form $\tM_{\not \emb}(U)$ where $U \in \disk(P)$ for some $P \in \tM$.
\end{enumerate}
\end{definition}

It will follow that the other types of sets mentioned above are also open; see \S \ref{sect:rectangular unions}.

\subsection{The topology on the canonical disk bundle}
\label{sect:topology total}
We recall that in Convention \ref{conv} we have identified the canonical disk bundle 
$\tE$ with the collection of pairs $(P,p)$, where $P$ is a planar surface and $p \in P$.

Let $P$ be a planar surface, let $K \in \PC(P)$ be a compact subset of $P$ with interior $K^\circ$. Let $U \subset K^\circ$ be an open set (not necessarily connected or containing the basepoint). Define the following subset of
$\tE$:
\begin{equation}
\label{eq:E+}
\tE_\imm(K,U)= \{(Q,q) \in \tE~:~\text{$\exists \iota:K \imm Q$ and $q \in \iota(U)$}\}.
\end{equation}
The {\em immersive topology on $\tE$} is the coarsest topology so that the projection $\tpi:\tE \to \tM$ is continuous and so that for each $P \in \tM$ and each $K \in \cdisk(P)$ and each open $U \subset K^\circ$, the set $\tE_\imm(K,U)$ is open.

The following results follow quickly from the definition:

\begin{proposition}
\label{prop:same topology}
Let $P$ be a planar surface. Then $P$ inherits a topology by viewing $P$ as a subset of $\tE$ with the immersive topology. This topology is the same as the topology coming from viewing $P$ as homeomorphic to an topological disk as in \S \ref{sect: fibers}.
\end{proposition}
\begin{proof}
Let $\sT$ be the first topology on $P$ mentioned. Then open sets in $\sT$ are of the form $\iota(U)$ where $\iota:K \to P$ is an immersion of a compact set. Since immersions are local homeomorphisms, $\iota(U)$ is always open in the disk topology. On the other hand if $V$ is an open set in the disk topology on $P$ and $v \in V$, then we can find an compact set $K \subset P$ containing $v$ in its interior, and we we can let $U=K^\circ \cap V$. Then the intersection
of $\tE_\imm(K,U)$ with $P$ is precisely $U$. Since this can be done for any $v \in V$, this shows $V$ is open in $\sT$.
\end{proof}

Sets $\tE_\imm(K,U)$ are open even if the condition that $K$ be a closed disk is removed:
\begin{proposition}
For every compact $K$ in some $P$, and every $U \subset K^\circ$, we have $\tE_\imm(K,U)$ is open.
\end{proposition}
\begin{proof}
Let $(Q,q) \in \tE_\imm(K,U)$. Then there is an immersion $\iota:K \imm Q$ and $q \in \iota(U)$.
Choose $D \in \cdisk(Q)$ so that $\iota(K) \subset D^\circ$. Then $\tE_\imm\big(D,\iota(U)\big)$ is open by definition and contains $(Q,q)$. We claim that  $\tE_\imm\big(D,\iota(U)\big) \subset \tE_\imm(K,U)$ which will prove that $\tE_\imm(K,U)$ is open. Let $(R,r) \in \tE_\imm\big(D,\iota(U)\big)$. Then there is an immersion $j:D \imm R$ and $r \in j \circ \iota(U)$. 
By composition we have an immersion $j \circ \iota:K \imm R$ and since $r \in j \circ \iota(U)$ we know 
$(R,r) \in \tE_\imm(K,U)$.
\end{proof}
\section{The Hausdorff property}
\label{sect:basic}
We prove the following Lemma:

\begin{lemma}
The immersive topologies on $\tM$ and $\tE$ are Hausdorff topologies.
\end{lemma}

\subsection{Basic properties of immersions}
\label{sect:basic immersions}
We collect some basic properties of immersions and embeddings:
\begin{proposition}
\label{prop:supremum}
Let $P$ and $Q$ be planar surfaces and $B \in \PC(Q)$. Suppose that $\langle A_j \in \PC(P) \rangle_{j \in \N}$ 
is an increasing sequence of open subsets, i.e. $A_j \subset A_{j+1}$ for all $j \in \N$. 
Let $U= \bigcup_{j \in \N} A_j \in \PC(P)$. Then:
\begin{itemize}
\item If $A_j \imm B$ for all $j \in \N$, then $U \imm B$. 
\item If $A_j \emb B$ for all $j \in \N$, then $U \emb B$. 
\end{itemize}
\end{proposition}
\begin{proof}
Suppose $\exists \iota_j:A_j \imm B$ for all $j \in \N$. The fact that these immersions are unique implies that for all $j<k$ and all $p \in A_j$, we have $\iota_j(p)=\iota_k(p)$,
since $\iota_j=\iota_k|_{A_j}$. Therefore, we may define a limiting map $\iota:U \to B$ by $\iota(p)=\iota_j(p)$ whenever $p \in A_j$. The preceding argument indicates that this map is well defined. We must check that it is an immersion. 
Because the $A_j$ are open, continuity of each $\iota_j$
implies continuity of $\iota$.
Similarly, $\iota$ acts as a translation in local coordinates by restricting locally within some $A_j$. 
The embedding case also requires checking injectivity. To prove this from the fact that each $\iota_j$ is injective, choose any distinct $p,q \in U$. Then $p,q \in \iota_j(A_j)$ for some $j$. Then by definition of $\iota$ and injectivity of $\iota_j$,
$$\iota(p)=\iota_j(p) \neq \iota_j(q)=\iota(q).$$
\end{proof}

\subsection{Subsets homeomorphic to disks}
\label{sect:balls}
In order to work with disks in planar surfaces, we will utilize some structure coming from Schoenflies' theorem:

\begin{theorem}[Schoenflies]
Let $C$ be a simple closed curve in the open topological disk $\Delta$. Then, there is a homeomorphism $h:\Delta \to \R^2$ so that $h(C)$ is the unit circle in $\R^2$. 
\end{theorem}

We translate this theorem into our setting as follows. 

\begin{corollary}
\label{cor:Schoenflies}
Let $P$ be a planar surface and $K \in \cdisk(P)$. 
Then, there is a homeomorphism $h:P \to \R^2$ so that $h(K)$ is the closed unit ball
and $h(o_P)=\0$.
\end{corollary}

We use this to impart the following structure.

\begin{proposition}
\label{prop:closed disk family}
Let $P$ be a planar surface. For each set $K \in \cdisk(P)$, there is a family of sets $\{K_t \in \cdisk(P)~:~t>0\}$ so that the following statements hold.
\begin{enumerate}
\item $K_1=K$.
\item $\bigcap_{t} K_t=\{o_P\}.$ 
\item $P=\bigcup_{t} K_t^\circ$.
\item For each $t>0$,  $K_t^\circ=\bigcup_{t'<t} K_{t'}^\circ$.
\label{item:intersection}
\item For each $t>0$, $K_t=\bigcap_{t'>t} K_{t'}$.
\label{item:union}
\item There is a continuous surjective function $\alpha:\R/2\pi\Z \times [0, \infty) \to P$, which is injective except that $\alpha(\R/2 \pi \Z \times \{0\})=\{o_P\}$ and satisfies $\alpha(\R/2 \pi \Z \times \{t\})=\partial K_t$.
\label{item:polar}
\end{enumerate}
\end{proposition}
\begin{proof}
Let $h:P \to \R^2$ be the homeomorphism guaranteed to exist by Corollary \ref{cor:Schoenflies}. The family given by 
$K_t=h^{-1} \big(\set{\v \in \R^2}{$\|\v\| \leq t$}\big)$ satisfies the proposition. The function $\alpha$ can be taken to be the pull back of polar coordinates on $\R^2$.
\end{proof}

We will call any family of sets $\{K_t \in \cdisk(P)~:~t>0\}$ formed as above a {\em closed disk family} in the planar surface $P$. Closed disk families give a natural way to understand immersions and the failure to immerse, and we use them throughout the paper. The next subsection uses them to prove our convergence criteria.

\subsection{Proofs of the Hausdorff property}

\begin{proof}[Proof that the immersive topology on $\tM$ is Hausdorff]
Let $P, Q \in \tM$ be distinct planar surfaces.
Then by Corollary \ref{cor:partial ordering}, either $P \not \imm Q$ or $Q \not \imm P$. Without loss of generality, assume $P \not \imm Q$. Let 
$\{K_t\}$ be a closed disk family for $P$. If each $K_t^\circ \imm Q$, then $P \imm Q$ by Proposition \ref{prop:supremum}.
Therefore, there is a $t$ so that $K_t^\circ \not \imm Q$. We conclude that $P \in \tM_{\imm}(K_t)$ and $Q \in \tM_{\not \imm}(K_t^\circ)$.
These open sets are disjoint since if $K_t \imm R$
then restriction gives an immersion $K_t^\circ \imm R$.\end{proof}

\begin{proof}[Proof that the immersive topology on $\tE$ is Hausdorff]
Suppose $(P,p)$ and $(Q,q)$ are distinct points in $\tE$. We will find open sets separating 
these points. If $P$ and $Q$ are distinct planar surfaces, then we can use the Hausdorff property of the embedding topology on $\tM$ to separate $P$ and $Q$ by open sets. The preimages of these open sets under $\tpi$ are open in $\tE$ and separate our points. 

Otherwise, we have $P=Q$, and $q \in P$. In this case, the points can be separated since the induced
topology on the planar surface $P$ is the same as the topology of an open topological disk; see Proposition \ref{prop:same topology}.
\end{proof}

\section{Fusing planar surfaces} 
\label{sect:fusion}

By Corollary \ref{cor:partial ordering} the notion of immersion places a partial ordering on the space $\tM$ of planar surfaces. We will now describe some of the order structure. If $P \imm Q$ we think of $Q$ as larger than $P$. The content of the following is that every collection of planar surfaces has a least upper bound:

\begin{theorem}[Fusion Theorem]
\label{thm:main fusion}
Let $\sP$ denote any non-empty collection of planar surfaces. Then there is a unique planar surface $R$ which satisfies the following statements:
\begin{enumerate}
\item[(I)] For each $P \in \sP$, $P \imm R$.
\item[(II)] For all planar surfaces $Q$, if 
$P \imm Q$ for all $P \in \sP$, then $R \imm Q$.
\end{enumerate}
\end{theorem}
We say that the planar surface $R$ from the above theorem is the {\em fusion} of
$\sP$ and write $R=\Fuse \sP$. If $\sP$ is a finite collection,
such as $\sP=\{P_1,P_2,P_3\}$, we write
$P_1 \fuse P_2 \fuse P_3$ for the fusion.
See Figure \ref{fig:fuse} for an example of this operation.

\begin{figure}
\includegraphics[width=3in]{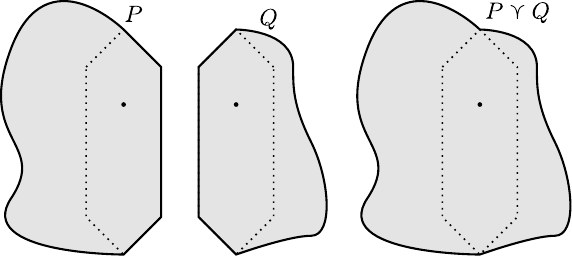}
\caption{An example of fusing two planar surfaces.}
\label{fig:fuse}
\end{figure}

We will now look at greatest lower bounds. Unfortunately, there is no minimal element in $\tM$. 
We can fix this by adding a new element; define $\bar \sM=\tM \cup \{O\}$
and extend the partial order $\imm$ by saying $O \imm P$ for all $P \in \bar \sM$ and $P \imm O$ if and only if $P=O$.
Now $O$ is the minimal element of $\bar \sM$. All subcollections of $\bar \sM$ have greatest lower bounds:

\begin{corollary}[Core Corollary]
\label{cor:core}
Let $\sP \subset \bar \sM$ be non-empty. Then, there is a unique $R \in \bar \sM$ which satisfies the following statements:
\begin{itemize}
\item[(I')] For each $P \in \sP$, $R \imm P$. 
\item[(II')] If $Q \in \bar \sM$ and $Q \imm P$ for all $P \in \sP$, then $Q \imm R$. 
\end{itemize}
\end{corollary}
The proof is a standard observation in order theory, but we give it for completeness. This and the result above indicate that $\bar \sM$ is a complete lattice in the sense of order theory. 
See \cite{Birkhoff64} or \cite{Gratzer11}
for background on complete lattices.
\begin{proof}
First we address uniqueness. Suppose $R_1$ and $R_2$ satisfy both statements.
Then by (I'), $R_1 \imm P$ and $R_2 \imm P$ for all $P \in \sP$. Then by (II'),
$R_1 \imm R_2$ and $R_2 \imm R_1$. We conclude that $R_1=R_2$ since $\imm$ is a partial order.

Consider the collection $\sS=\{S \in \tM~:~\text{$S \imm P$ for all $P \in \sP$}\}.$ If this collection is empty,
then we can take $R=O$. The statements (I') and (II') are clearly satisfied.

If $\sS \neq \emptyset$, then let $R=\Fuse \sS$. We will now prove that (I') is satisfied.
Fix $P \in \sP$. Observe that $S \imm P$ for all $S \in \sS$. Therefore $R \imm P$ by statement (II) of the Fusion theorem. We now prove that (II') is satisfied. Suppose $Q \in \bar \sM$ and suppose $Q \imm P$ for every $P \in \sP$. 
Then $Q \in \sS$. So $Q \imm R$ by statement (I) of the Fusion Theorem.
\end{proof}
We call the $R \in \bar \sM$ produced in the above corollary the {\em core} of the collection $\sP$,
and denote this by $R=\Core \sP$. 


\subsection{Trivial structures}
\label{sect:trivial structures}
Let $\Sigma$ be a connected oriented topological surface with basepoint $x_0$. We will say that a {\em trivial structure} on a surface is an atlas of orientation preserving local homeomorphisms (charts) to the plane so that the transition functions are restrictions of the identity map on the plane and so that the image of the basepoint $x_0$ is always mapped to $\0 \in \R^2$. We emphasize that in a trivial structure, the underlying topological surface $\Sigma$ need not be a topological disk.

We give some examples of trivial surfaces to orient the reader. Any connected open subset of $\R^2$ containing the origin $\0$ is a trivial surface, where it suffices to take the identity map as the sole chart in the atlas. (See Figure \ref{fig:trivial}.)
Considering the developing map as the sole chart for a planar surface yields a trivial structure on any planar surface. A trivial structure is a special type of translation structure, i.e., for every translation structure there is at most one corresponding trivial structure. Furthermore, a translation structure on a surface $S$ admits a compatible trivial structure if and only if the holonomy map $\pi_1(S) \to \R^2$ has trivial image. So, in particular, given any translation structure on a surface, there is a compatible trivial structure on its universal abelian cover.

\begin{figure}
\includegraphics[width=5in]{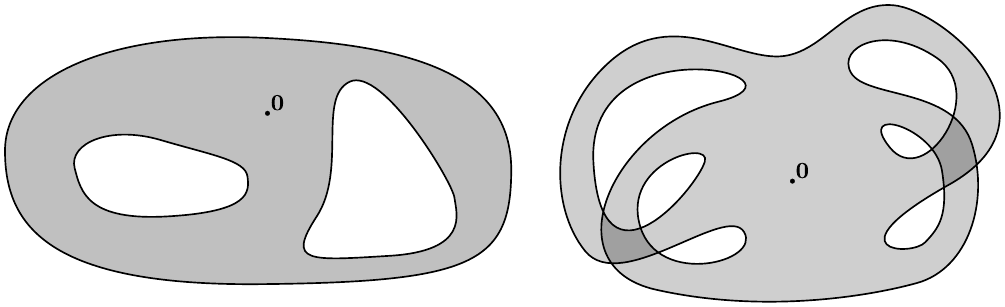}
\caption{Trivial structures on the thrice punctured sphere and once punctured genus two surface.}
\label{fig:trivial}
\end{figure}

We will see that trivial structures generalize but share a lot of properties of translation structures on disks.
Fix a translation structure on a pointed topological surface $\Sigma$. 
Because the transition functions are the identity map, the image of a point $y \in \Sigma$ under a chart is independent of the choice of the chart. It follows that a trivial structure on a surface can be specified by a single chart, a single orientation preserving local homeomorphism $\phi:\Sigma \to \R^2$ so that $\phi(x_0)=\0$ (which generalizes the developing map for a translation structure on a disks). Two such local homeomorphisms $\phi$ and $\psi$ are said to yield {\em isomorphic trivial structures} if there is an orientation preserving homeomorphism $h:\Sigma \to \Sigma$ so that $h(x_0)=x_0$ and $\phi \circ h^{-1}=\psi$. The homeomorphism $h$ is called an {\em isomorphism} between the structures. The {\em (set-theoretic) moduli space of trivial structures on $(\Sigma,x_0)$} is the 
collection isomorphism-equivalence classes of trivial structures on $(\Sigma,x_0)$.
We can construct a canonical $\Sigma$-bundle over the moduli space of trivial structures on $(\Sigma,x_0)$,
as for translation structures on the disk. We say that a { \em trivial surface (homeomorphic to $(\Sigma,x_0))$} is a fiber of the projection from the bundle to the moduli space. We endow the trivial surfaces with a pointed trivial structure in the canonical way. This choice makes the trivial structure on the fiber an element of the isomorphism class described by the image in moduli space.

The collection of all trivial surfaces is the collection of fibers of such projections taken over all
homeomorphism-equivalence classes of pointed surfaces $(\Sigma,x_0)$. The developing map $\dev$ is well defined
on the union of all trivial surfaces, and the restriction to a single trivial surface is a local homeomorphism to $\R^2$ which carries the basepoint to $\0$. 

The notions of immersion and embedding carry over trivially to subsets of trivial surfaces.
In particular, we note that if $P$ and $Q$ are trivial surfaces, then $P \imm Q$ and $Q \imm P$ implies that $P=Q$. 

\subsection{The Fusion Theorem}
\label{sect:fusion theorem}
The main goal of this section is to prove the following theorem.

\begin{theorem}[Generalized Fusion Theorem]
\label{thm:generalized fusion}
Let $\sP$ denote any non-empty collection of trivial surfaces. Then there is a unique trivial surface $R$ which satisfies the following statements:
\begin{enumerate}
\item[(I)] For each $P \in \sP$, $P \imm R$.
\item[(II)] For all planar surfaces $Q$, if 
$P \imm Q$ for all $P \in \sP$, then $R \imm Q$.
\end{enumerate}
\end{theorem}
We call $R$ the fusion of the surfaces in $\sP$, and use notation for $R$ as described under the statement of Theorem \ref{thm:main fusion}. 

We note that $\tM$ is by definition the (set-theoretic) moduli space of trivial structures on the pointed disk $(\Delta,x_0)$. As such $\tM$ is a subset of the (set-theoretic) moduli space of trivial structures on all pointed surfaces defined as defined in \S \ref{sect:trivial structures}. Therefore, we can deduce
our original Fusion Theorem (Theorem \ref{thm:main fusion}) from the Generalized Fusion Theorem using the following.
\begin{proposition}
\label{prop:fusion disks}
If $\sP$ is a collection of trivial surfaces homeomorphic to disks, then so is $\Fuse \sP$. 
\end{proposition}
\begin{proof}
Suppose to the contrary that $R=\Fuse \sP$ is not simply connected. 
Let $\tilde R$ be the universal cover of $R$ and let $\pi:\tilde R \to R$ be the covering.
We will prove that $\tilde R$ also satisfies statements (I) and (II) of the Fusion theorem,
contradicting uniqueness unless $\tilde R=R$. 

Since each $P \in \sP$ is simply connected, the immersion $P \imm R$ lifts to an immersions $P \imm \tilde R$. Thus (I) is satisfied by $\tilde R$. 
Now suppose that $P \imm S$ for all $P \in \sP$. By statement (II) for $R$, we know $R \imm S$. The covering map
$\tilde R \to R$ is an immersion, so by composition with the covering map, $\tilde R \imm S$. Thus, (II) is satisfied.
\end{proof}

For the rest of the section we will work on proving the Generalized Fusion Theorem, and we will only study trivial surfaces. We prove this theorem in a series of steps starting with uniqueness (assuming existence).

\begin{proof}[Proof of uniqueness in the Generalized Fusion Theorem]
Suppose there are two trivial surfaces $R_1$ and $R_2$ which satisfy statements (I) and (II) of the Theorem. Then 
by statement (I), for each $j \in \{1,2\}$ and each $P \in \sP$, we have
$P \imm R_j$. Then by statement (II), we
have $R_1 \imm R_2$ and $R_2 \imm R_1$. So $R_1=R_2$ by Corollary \ref{cor:partial ordering},
which may be seen to hold for trivial surfaces. (The same proof works as for planar surfaces.)
\end{proof}

\subsection{Construction of the fusion}
\label{sect:constructive fusion}
We construct the fusion of $\sP$ as a quotient of the disjoint union $\bigsqcup_{P \in \sP} P$. 
We make this a topological space by making each open set in any $P \in \sP$ open in the disjoint union.

The bundle-wide developing map $\dev$ is defined on the union of all trivial surfaces,
so by restriction, it is also defined on the disjoint union $\bigsqcup_{P \in \sP} P$.

Let $\sim$ be an equivalence relation on $\bigsqcup_{P \in \sP} P$. Make the following choices:
\begin{itemize}
\item Let $P_1$ and $P_2$ be a pair of (not necessarily distinct) surfaces  in $\sP$.
\item Let $r_1 \in P_1$ and $r_2 \in P_2$ be so that $r_1 \sim r_2$.
\item Let $\gamma_1:[0,1] \to P_1$ and $\gamma_2:[0,1] \to P_2$ be paths so that
$\gamma_1(0)=r_1$, $\gamma_2(0)=r_2$ and  
$$\dev \circ \gamma_1(t)=\dev \circ \gamma_2(t) \quad \text{for all $t \in [0,1]$}.$$
\end{itemize}
We say that $\sim$ is {\em path invariant} if for every choice made as above,
we have $\gamma_1(1) \sim \gamma_2(1)$.  

\begin{theorem}[Constructive Fusion Theorem]
\label{thm:constructive fusion}
Let $\sim$ be the smallest path invariant equivalence relation on $\bigsqcup_{P \in \sP} P$ so that 
for each pair of surfaces $P$ and $Q$ in $\sP$, we have $o_P \sim o_Q$. Then, $\bigsqcup_{P \in \sP} P / \sim$ is isomorphic to the fusion of $\sP$. Here, the local homeomorphism from $\bigsqcup_{P \in \sP} P / \sim$
to $\R^2$ is provided by
the developing map, which descends to this quotient.
\end{theorem}

We establish a corollary to the Constructive Fusion Theorem, which reveals some structure of the fusion of infinitely many surfaces which is unapparent in the original Fusion Theorem.

\begin{corollary}
\label{cor:fusion finiteness}
Let $\sP$ be an infinite collection of trivial surfaces. Let $R=\Fuse \sP$. For $P \in \sP$, let $\iota_P:P \imm R$ be the immersion guaranteed by statement (I) of the Fusion Theorem. Let $p$ be a point in $P \in \sP$ and $q$ be a point in $Q \in \sP$, and suppose that $\iota_P(p)=\iota_Q(q)$. Then there is a finite subcollection $\sF \subset \sP$ containing $P$ and $Q$ so that the immersions $j_P:P \imm \Fuse \sF$ and $j_Q:Q \imm \Fuse \sF$ satisfy $j_P(p)=j_Q(q)$.
\end{corollary}
\begin{proof}
Let $\sim$ be the equivalence relation from the Constructive Fusion Theorem.
We think of equivalence relations on $\bigsqcup_{P \in \sP} P$ as subsets of $(\bigsqcup_{P \in \sP} P)^2$. 
We will construct an increasing sequence of equivalence relations $\sim_n$ on $\bigsqcup_{P \in \sP} P$ so that $\bigcup_n \sim_n=\sim$.
Then, the finiteness result follows if the finiteness result is proved for each $\sim_n$.

We define $\sim_n$ inductively in the integers $n \geq 0$ beginning with $\sim_0$. Let $p,q \in \bigsqcup_{P \in \sP} P$. We define 
$p \sim_0 q$ if $p=q$ or if $p$ and $q$ are both basepoints of surfaces in $\sP$. This can be seen to be an equivalence relation.
Now suppose that $\sim_n$ is defined and let $p \in P \in \sP$ and $q \in Q \in \sP$ be points in $\bigsqcup_{P \in \sP} P$. 
We say $p$ is $n+1$-related to $q$ (denoted $p \equiv_{n+1} q$) if there are curves $\gamma_1:[0,1] \to P$
and $\gamma_2:[0,1] \to Q$ so that $\gamma_1(0) \sim_{n} \gamma_2(0)$, $\gamma_1(1)=p$, $\gamma_2(1)=q$ and $\dev \circ \gamma_1(t)=\dev \circ \gamma_2(t)$ for all $t \in [0,1]$. Observe that $\sim_n \subset \equiv_{n+1}$. We define $\sim_{n+1}$ to be the smallest equivalence relation containing $\equiv_{n+1}$. Since
$\equiv_{n+1}$ is reflexive and symmetric, we can concretely say that $p \sim_{n+1} q$ if $p \equiv_{n+1} q$ or if there is a finite collection $p_1, p_2,\ldots, p_k \in \bigsqcup_{P \in \sP} P$ so that the following holds:
\begin{equation}
\label{eq:relation}
p \equiv_{n+1} p_1 \equiv_{n+1} p_2 \equiv_{n+1} \ldots \equiv_{n+1} p_k \equiv_{n+1} q.
\end{equation}
Observe that by definition of $\sim$ we have $\bigcup_n \sim_n=\sim$.

We now prove our finiteness statement by induction. Let $p \in P \in \sP$ and $q \in Q \in \sP$. If $p$ and $q$ are the same point, then $P=Q$,
we can take $\sF=\{P\}$ so that 
$P=\Fuse \sF$, and the identity map $P \imm P$ sends $p$ and $q$ to the same point. If $p$ and $q$ are basepoints of $P$ and $Q$, respectively, then 
the immersions $P \imm (P \fuse Q)$ and $Q \imm (P \fuse Q)$ carry these points to the basepoint of $P \fuse Q$ by definition of immersion. This proves the finiteness statement for $\sim_0$.

Now suppose $\sim_n$ satisfies the finiteness statement, and suppose that $p \equiv_{n+1} q$. Then, there must be curves $\gamma_1:[0,1] \to P$ and $\gamma_2:[0,1] \to Q$ as above. Then $\gamma_1(0) \sim_{n} \gamma_2(0)$, so there is a finite collection $\sF \subset \sP$ containing $P$ and $Q$ so that the immersions
$\iota_P:P \imm \Fuse \sF$ and $\iota_Q:P \imm \Fuse \sF$ satisfy $\iota_P \circ \gamma_1(0)=\iota_Q \circ \gamma_2(0)$. Then by the Constructive Fusion Theorem applied to $\Fuse \sF$, we see that $\iota_P \circ \gamma_1(1)=\iota_Q \circ \gamma_2(1)$ as well. This proves the finiteness statement for $\equiv_{n+1}$.

Now suppose that $\equiv_{n+1}$ satisfies the finiteness statement, and suppose that $p \sim_{n+1} q$. Then there are points  $p_1, p_2,\ldots, p_k \in \bigsqcup_{P \in \sP} P$ satisfying equation \ref{eq:relation}. Let $p_0=p$ and $p_{k+1}=q$. Let $P_j \in \sP$ be the surface containing $p_j$ for each $j$. 
Then for all $j \in \{0,\ldots, k\}$, there is a finite collection 
$\sF_j$ so that the immersions $\iota_j:P_j \imm \Fuse \sF_j$ and $\iota_j':P_{j+1} \imm \Fuse \sF_j$ satisfies $\iota_j(p_j)=\iota_j'(p_{j+1})$. 
Let $\sF=\bigcup_{j=0}^k \sF_j$. Then, there are immersions $\jmath_j:\Fuse \sF_j \imm \Fuse \sF$ for all $j$. The immersions
$P_j \imm \Fuse \sF$ can be given by $\jmath_j \circ \iota_j$ for $j \leq k$ and by $\jmath_{j-1} \circ \iota'_{j-1}$ for $j \geq 1$. 
It follows that the image of $p_j$ inside $\Fuse \sF$ is independent of $j \in \{0,\ldots,k+1\}$. 
\end{proof}

\subsection{The fusion is a trivial surface}
We will begin by proving that the quotient space described in the Constructive Fusion Theorem
is really a trivial surface. 

\begin{lemma}
Let $\sim$ be the equivalence relation from the Constructive Fusion Theorem.
Then, the quotient $\bigsqcup_{P \in \sP} P/\sim$ has the structure of a trivial surface
with the associated immersion to $\R^2$ given by 
$$\phi([r])=\dev(r) \quad \text{for all $r \in P \in \sP$}.$$
\end{lemma}
We will devote the remainder of the section to the proof of this fact. We now describe our plan. We will check that $\phi$ is a well defined map. Then we will prove that $\phi$ is a local homeomorphism. This demonstrates that 
$\bigsqcup_{P \in \sP} P/\sim$ is locally modeled on $\R^2$. Finally, we will show that $\bigsqcup_{P \in \sP} P/\sim$ is Hausdorff. This proves that it is a surface, and $\phi$ gives this surface a trivial planar structure.

\begin{proof}[Proof that $\phi$ is well defined]
The basepoint of $\bigsqcup_{P \in \sP} P/\sim$ is given by the equivalence class $[o_P]$
for some (any) $P \in \sP$. We note that the
developing map sends $o_P$ to zero. By induction, we can see that the points we are forced to identify by path invariance (see the definition above Theorem \ref{thm:constructive fusion}) also have the same image under the developing map. Therefore, $\phi$ is well defined.
\end{proof}

In order to prove the remainder of the lemma, it is useful to use the following:
\begin{proposition}
\label{prop:open}
Let $r_1 \in P_1$ for some $P_1 \in \sP$. Let $B \subset P_1$ be an open metric ball of radius $\epsilon$ about $r_1$ with $\epsilon$ taken to be so small that the developing map restricted to $B$ gives a homeomorphism to a ball of radius $\epsilon$ in the plane. Let 
$$U=\{r_2 \in \bigsqcup_{P \in \sP} P~:~ \text{$r_2 \sim b$ for some $b \in B$}\}.$$
Then, $U$ is open in $\bigsqcup_{P \in \sP} P$, and so by definition 
$B'=\{[b]~:~b \in B\}$ is open in $\bigsqcup_{P \in \sP} P/\sim.$
\end{proposition}
\begin{proof}
We remind the reader of the topology we placed on $\bigsqcup_{P \in \sP} P$.
We need to show that $U \cap P_2$ is open in $P_2$ for all $P_2 \in \sP$. 
Let $r_2 \in U \cap P_2$. Then, $r_2 \sim b$ for some $b \in B$.
Since the image under the developing map is an $\sim$-invariant, we have
$$|\dev(r_1)-\dev(r_2)|=|\dev(r_1)-\dev(b)| < \epsilon,$$
because $b$ lies in the ball $B$ of radius $\epsilon$ about $r_1$. Choose 
$\epsilon'>0$ small enough so that the following hold:
\begin{enumerate}
\item The developing map restricted to the open $\epsilon'$-ball about $r_2$ is a homeomorphism to a ball in the plane of radius $\epsilon'$.
\item The $\epsilon'$-ball about $b$ is a subset of the ball $B$.
\end{enumerate}
Let $D$ denote the open ball of radius $\epsilon'$ about $r_2$ in $P_2$. Let $r_3 \in D$. Then by (1) there is a path $\gamma_1$ of length less than $\epsilon'$ joining $r_2$ to $r_3$. 
Similarly, there is a path $\gamma_2$ starting at $b$
and contained in $B$ so that $\dev \circ \gamma_1=\dev \circ \gamma_2$. 
This path stays within $B$ by (2).
Thus $U \cap P_2$ is open as desired.
\end{proof}

\begin{proof}[Proof that $\phi$ is a local homeomorphism]
Choose any $[r_1] \in \bigsqcup_{P \in \sP} P/\sim$, and choose $r_1 \in [r_1]$. Let $U$ be as in the above proposition. Then the set $B'$ at the end of the proposition is an open set containing $[r_1]$. Furthermore, since each point in $B'$
has a unique representative in $B$, we know that $\phi|_{B'}$
is one-to-one and onto its image, which is an open ball in $\R^2$. 

We must prove that $\phi|_{B'}$ is continuous. 
This also follows from the proposition. Let $\v$ be a point in $\phi(B')$, and let $[r_2]=\phi^{-1}(\v)$. We
can choose the representative $r_2 \in [r_2] \cap B$. 
Then the neighborhood of radius $\epsilon-|\dev(r_1)-\dev(r_2)|$ about $r_2$ is isometric to a Euclidean ball. Applying the proposition to this choice
of center $r_2$ and radius produces an open set containing $r_2$ and contained in $B'$. 

The fact that $(\phi|_{B'})^{-1}$ is continuous is a tautology,
because of the topology we placed on $\bigsqcup_{P \in \sP} P$.
Recall that the union of equivalence classes in
$B'$ is open. Call this union $U$ as in the lemma above. Now let $C' \subset B'$ be a smaller open set. This by definition means that its union of equivalence classes $V' \subset U'$ is open.
That is, $V' \cap P$ is open for each $P \in \sP$. 
Moreover by definition of $\phi$, we have
$$\phi|_{B'}(C')=\dev(V')=\bigcup_{P \in \sP} \dev|_{P} (V' \cap P).$$
But, the image of any open set in a trivial surface under the developing map is open in $\R^2$,
and any union of open sets is open.
\end{proof}

\begin{proof}[Proof that $\bigsqcup_{P \in \sP} P/\sim$ is Hausdorff.]
Let $[r_1], [r_2] \in \bigsqcup_{P \in \sP} P/\sim$ be distinct. We will separate these points by open sets. First suppose that
$\phi([r_1]) \neq \phi([r_2])$. Then by constructing neighborhoods around each of $[r_1]$ and $[r_2]$ using Proposition \ref{prop:open} with radius less than or equal to 
$\frac{1}{2} |\phi([r_1]) - \phi([r_2])|$ produces open sets which can discerned to be disjoint because their images under $\phi$ are disjoint.

Now suppose that $[r_1]$ and $[r_2]$ are distinct but that $\phi([r_1]) = \phi([r_2])$. Choose representatives $r_1 \in [r_1]$ and $r_2 \in [r_2]$. 
Suppose $r_1 \in P_1$ and $r_2 \in P_2$. Let $B_1 \subset P_1$ and $B_2 \subset P_2$ be open metric balls about $r_1$ and $r_2$, respectively, which are each isometric to a Euclidean metric ball. These balls determine open sets $B_1', B_2' \subset \bigsqcup_{P \in \sP} P/\sim$ by Proposition \ref{prop:open}. We claim that they are disjoint. Suppose not. Then, there is a $[r_3] \in B_1' \cap B_2'$. By the proposition, we can then find 
points $b_1 \in B_1 \cap [r_3]$ and $b_2 \in B_2 \cap [r_3]$. Since $b_1 \sim b_2$, they have the same image under $\dev$. Parameterize the line segments joining $b_1$ to $r_1$ within $B_1$ and joining $b_2$ to $r_2$ within $B_2$ in the same way. Then, path invariance guarantees that $r_1 \sim r_2$. This contradicts the distinctness of $[r_1]$ and $[r_2]$.
\end{proof}

\subsection{Proof of the Fusion Theorem}
We now prove the Constructive Fusion Theorem. Note that this immediately implies the Generalized Fusion Theorem (Theorem \ref{thm:generalized fusion}), and the original version of the Fusion Theorem (Theorem \ref{thm:main fusion})
follows from Proposition \ref{prop:fusion disks}.
\begin{proof}[Proof of Theorem \ref{thm:constructive fusion}]
Let $P \in \sP$. Let $R=\bigsqcup_{P \in \sP} P / \sim$, where $\sim$ is equivalence relation described in the theorem. We will prove that $R$ has the properties described in the Fusion Theorem. 

Statement (I) of Theorem \ref{thm:generalized fusion} simply requires proving that the natural maps $P \to R$ respect the basepoints, respect the developing maps, and are local homeomorphisms. Basepoints are respected by construction. By definition of $\phi$, the developing map is respected. This proves that the natural map $P \to R$ is an immersion. Finally, the fact that $P \to R$ is a local homeomorphism follows from the fact that the developing maps are respected and are local homeomorphisms.

Statement (II) of Theorem \ref{thm:generalized fusion} reduces to a statement about equivalence relations. Suppose $P \imm S$ for all $P \in \sP$. 
Let $j:\bigsqcup_{P \in \sP} P \to S$ be the simultaneous immersion of all planar surfaces $P \in \sP$ into $S$. 
Then, we define an equivalence relation on $\bigsqcup_{P \in \sP} P$ by $p \approx q$ 
for $p \in P \in \sP$ and $q \in Q \in \sP$ if $j(p)=j(q)$. 
Then all basepoints are equivalent and $\approx$ is path invariant. 
Since $\sim$ is the smallest such relation, each $\sim$-equivalence class is contained in an $\approx$-equivalence class. This gives a canonical map $R \to S$. By construction, it is an immersion.
\end{proof}

\section{New open sets and second-countability}
\label{sect:rectangular unions}
In this section, we will study subsets of planar surfaces which are unions of rectangles. Analyzing such sets will enable us to prove two important results about the immersive topology.

\begin{theorem}[Open sets in $\tM$]
\label{thm:open}
Let $P$ be a planar surface. If $K \in \PC(P)$ is compact,
then both $\tM_{\imm}(K)$ and $\tM_{\emb}(K)$ are open in the immersive topology on $\tM$. If $U \in \PC(P)$ is open, then
both $\tM_{\not \imm}(U)$ and $\tM_{\not \emb}(U)$ are open.
\end{theorem}

\begin{theorem}
\label{thm:2nd countable}
The immersive topologies on $\tM$ and $\tE$ are second-countable, that is, they admit a countable basis.
\end{theorem}
We remark that Propositions \ref{prop:second countability} and \ref{prop:second countability2} will give explicit countable subbases for these topologies. (By general principles, the induced bases are then also countable.)

\subsection{Definition of rectangular union}
A closed {\em rectangle in the plane} is a subset of $\R^2$ of the form 
$$[a,b] \times [c,d]=\{(x,y) \in \R^2~:~
\text{$a \leq x \leq b$ and $c \leq y \leq d$}\},$$
where $a<b$ and $c<d$.
Similarly, an {\em open rectangle} is a set of the form $(a,b) \times (c,d)$.
We call such a rectangles {\em rational} if $a$, $b$, $c$ and $d$ are rational numbers.

Let $P$ be a planar surface. We call $R \subset P$ a {\em closed (resp. open) rectangle} if $\dev(R)$ is a closed (resp. open) rectangle and the restriction $\dev|_{R}:R \to \dev(R)$ is a homeomorphism. We say $R$ is {\em rational} if $\dev(R)$ is.

A closed (resp. open) {\em rectangular union} is a finite union of closed (resp. open) rectangles in a planar surface
which is connected and whose boundary is a disjoint collection of curves.
We call a rectangular union {\em rational} if it can be constructed as a finite union of rational rectangles.

\begin{proposition}
\note{Seems generally relevant}
The closure of an open rectangular union with compact closure is a closed rectangular union. The interior of a closed rectangular union is an open rectangular union.
\end{proposition}
\begin{proof}
Let $P$ be a planar surface.
Suppose $\{R_i \subset P\}$ is a finite collection of open rectangles whose union is an open rectangular union $U$ with compact closure.
Then $\bar U=\bigcup_i \bar R_i$ is a closed rectangular union.

Now suppose $\sR=\{R_i \subset P\}$ is a collection of closed rectangles whose union is a closed rectangular union $K$.
Then $\bigcup_i R_i^\circ$ may not be as large as $K^\circ$. 
We will construct a larger finite collection of closed rectangles $\sR' \supset \sR$ so that $K^\circ=\bigcup_{R \in \sR'} R^\circ$. We will describe an algorithm for constructing $\sR'$ by adding rectangles beginning with $\sR'=\sR$. 
Let $\Lambda=K^\circ \smallsetminus \bigcup_i R_i^\circ$. 
A point $p \in \Lambda$ is either a vertex of a rectangle in $\sR$ or lies in the interior of an edge of such a rectangle. If $p \in \Lambda$ is a vertex, then it has a neighborhood which lies in $K^\circ$, so we can add a small rectangle to $\sR$ which contains $p$ and is contained entirely in $K^\circ$. 
We add such a rectangle to $\sR'$ for each vertex in $\Lambda$. Now suppose $p \in \Lambda$ is not a vertex. Then, it must lie in the common boundary of two rectangles $R_1$ and $R_2$ lying on opposite sides of a line whose edges intersect in an interval. We can construct a  closed rectangle, $R_0$, which is contained in $R_1 \cup R_2$ and contains the overlap $R_1 \cap R_2$. See below:
\begin{center}
\includegraphics{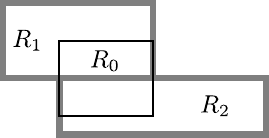}
\end{center}
We add such a rectangle to $\sR'$ for all pairs of rectangles in $\sR$ which intersect in an interval. The resulting $\sR'$ has the desired property.
\end{proof}

\begin{proposition}
\label{prop:open disk}
An open rectangular union is homeomorphic to a finitely punctured disk.
\end{proposition}
\begin{proof}
By definition, an open rectangular union is a connected surface in a topological disk. So, it is homeomorphic to a punctured disk, but the number of punctures may be infinite.
So, it suffices to  prove that an open rectangular union has finite Euler characteristic.

We will show that a {\em union of open rectangles}, i.e., an arbitrary union of open rectangles in a planar surface,
has finite Euler characteristic. We define the {\em complexity}
of such a union to be the smallest number of rectangles necessary to write the set as a union of rectangles.
In fact, we will prove that if a union of rectangles $U$ has complexity less than $n$, then
$|\chi(U)|<2^n$. For $n=0$, $\chi(U)=0$, and for $n=1$, $\chi(U)=1$. These provide a base case, and 
we proceed by induction in $n$. Suppose the statement $|\chi(U)|<2^n$ for all $U$ of complexity $n$. Let $U'$ be a union of complexity $n+1$. Then $U'=U \cup R$, where
$U$ has complexity $n$ and $R$ is another open rectangle. 
By the inclusion-exclusion principle,
\begin{equation}
\label{eq:e i}
\chi(U \cup R)=\chi(U)+\chi(R)-\chi(U \cap R)=\chi(U)-\chi(U \cap R)+1.
\end{equation}
Note that the intersection of two open rectangles in a planar surface is either empty or another open rectangle. In particular, $U \cap R$ is either empty or a union of rectangles of complexity no more than $n$. By inductive hypothesis, $|\chi(U)|<2^n$
and $|\chi(U \cap R)|<2^n$. So by equation \ref{eq:e i},
$$|\chi(U \cup R)| \leq |\chi(U)|+|\chi(U \cap R)|+1 \leq (2^n-1)+(2^n-1)+1<2^{n+1}.$$
\end{proof}

\subsection{A finiteness condition}
Let $P$ and $Q$ be planar surfaces. Let $A \in \PC(P)$ and $B \in \PC(Q)$. We say $A$ and $B$ are {\em isomorphic}
if $A \imm B$ and $B \imm A$. This defines an equivalence relation on 
$$\PC=\bigcup_{P \in \tM} \PC(P).$$
We note that sets of the form
$\tM_{\imm}(K)$, $\tM_{\emb}(K)$, $\tM_{\not \imm}(U)$ and $\tM_{\not \emb}(U)$ only depend on the isomorphism classes of $K$ and $U$.

We state our main finiteness result for convex sets rather than rectangles
because it will be useful in later sections.
\begin{proposition}
\label{prop:finite}
Let $A_1,\ldots, A_n$ be a collection of convex subsets of the plane 
which are either all open or all closed. There are only finitely many isomorphism classes of $U \in \PC$ which are the
union of sets of the form $\tilde A_i \subset U$ so that 
$A_i=\dev(\tilde A_i)$ for all $i$. 
\end{proposition}
The reason this proposition will be useful is that our topology on $\tM$ was defined in terms of a subbasis, and finite intersections of elements of the subbasis are still open but typically not elements of the subbasis. We eventually be taking intersections of sets of the form $\tM_{\imm}(K)$ and $\tM_{\not \emb}(U)$ over finite collections of isomorphism classes in order to prove Theorem \ref{thm:open}.

\begin{proof}
Fix $A_1,\ldots, A_n$. Let $U=\bigcup_{i=1}^n \tilde A_i$ be such a union. Suppose $U$ lives in the planar surface $P$ with basepoint $o_P$. 
We associate $U$ to two pieces of information. First there is a subset $\sS(U) \subset \{1,\dots,n\}$ consisting of those 
$i$ so that $o_P \in \tilde A_i$.
Also, we can associate to $U$ a subgraph $\sG(U)$ of the complete graph $K_n$ with vertex set $\{1,\ldots, n\}$. We
define this subgraph by the condition that there is an edge between distinct $i ,j \in \{1,\ldots, n\}$ if $\tilde A_i \cap \tilde A_j \neq \emptyset$. 
Assume this intersection is non-empty. By convexity of $A_i$ and $A_j$, it follows that $A_i \cap A_j$ is convex and so connected. Because the lifts
$\dev|_{A_i}^{-1}:A_i \to P$ and $\dev|_{A_j}^{-1}:A_j \to P$ agree at one point, they must agree at all point of $A_i \cap A_j$ by analytic continuation. Therefore, we can recover $U$ up to isomorphism from $\sS(U)$ and $\sG(U)$. Consider the disjoint union
$\bigsqcup_{i} A_i$. Inclusion of each $A_i$ into $\R^2$ gives a natural map $\pi:\bigsqcup_{i} A_i \to \R^2$. Define the equivalence relation $\sim$ on the disjoint union by $p \in A_i$ is equivalent to $q \in A_j$
if $\pi(p)=\pi(q)$ and the edge $\bar {ij}$ lies in $\sG(U)$. 
There is a natural identification between $U$ and
$\bigsqcup_{i} A_i/\sim$ which picks out the isomorphism class of $U$. The collection of points $p \in A_i$ with $i \in \sS(U)$ and 
$\dev(p)=\0$ is an equivalence class of $\bigsqcup_{i} A_i/\sim$. This corresponds to the basepoint of $U$. 

Let $U_1$ and $U_2$ be unions coming from the 
same choices of convex sets $A_1,\ldots, A_n$.
We remark that there is an immersion $\iota:U_1 \to U_2$
if and only if $\sS(U_1) \subset \sS(U_2)$ and 
$\sG(U_1)$ is a subgraph of $\sG(U_2)$. Viewing
$$U_1=\bigsqcup_{i} A_i/\sim_1 \and U_2=\bigsqcup_{i} A_i/\sim_2,$$
we observe that these conditions imply that the identity
map $\bigsqcup_{i} A_i \to \bigsqcup_{i} A_i$ induces
a well defined map from $U_1 \to U_2$. This is the needed immersion.

It follows that the collection of unions satisfying the proposition is finite: There are no more than the number of pairs $\big(\sS(U),\sG(U)\big)$,
where $\sS(U) \subset \{1,\dots,n\}$ and $\sG(U)$ is a subgraph of $K_n$.
\end{proof}

\begin{corollary}
\label{cor:finite immersed}
Let $U$ be an open (resp. closed) subset of a planar surface which
is a finite union of open (resp. closed) convex sets. Then there are only finitely many images of $U$ under immersions up to isomorphism. 
\end{corollary}
\begin{proof}
Let $U=\bigcup_{i=1}^n \tilde A_i$ be a union of convex sets in a planar surface $P$. Let $A_i=\dev(\tilde A_i)$. Given an immersion $\iota:U \imm Q$, we have
$\iota(U)=\bigcup_{i=1}^n \iota(\tilde A_i)$. This writes
$\iota(U)$ as a union of lifts of the sets $A_i \subset \R^2$ for $i=1,\ldots, n$. There are only finitely many possibilities for $\iota(U)$ by Proposition \ref{prop:finite}.
\end{proof}

\begin{corollary}
\label{cor:countably many classes}
\note{Definitely used.}
There are only countably many isomorphism classes of
(open or closed) rational rectangular unions in $\PC$. 
\end{corollary}
\begin{proof}
This follows from the proposition above, because there are only countably many finite collections of rational rectangles in the plane.
\end{proof}

\subsection{Rectilinear curves}

We will say a closed curve $\gamma:\R/L\Z \to \R^2$ is {\em rectilinear} if there are $0 = t_0<t_1<\ldots<t_{2k}=L$ so that the derivative satisfies
$$\gamma'(t) = \begin{cases}
(\pm 1,0) & \text{if $t_j<t<t_{j+1}$ with $j$ even,} \\
(0,\pm 1) & \text{if $t_j<t<t_{j+1}$ with $j$ odd.} \\
\end{cases}$$
We say the rectilinear curve $\gamma$ is {\em rational}
if the points $\gamma(t_j)$ are rational.

Our topology on $\tM$ was defined using a subbasis consisting of elements of the form $\tM_{\imm}(K)$ and
$\tM_{\not \emb}(U)$ where $K$ and $U$ are disks, and so we will need to study rectangular unions which are also topological disks. Such disks are bounded by rectilinear curves, and we will use an understanding of rectilinear curves to deduce certain finiteness results. 
\begin{lemma}
\label{lem:2}
\note{Used by Corollary \ref{cor:finite open disks}.}
Let $\gamma$ be a closed immersed rectilinear curve in $\R^2$. Then, lifts of
$\gamma$ to simple closed curves in planar surfaces bound at most finitely many isomorphism classes of disks,
and each such disk is a rectangular union. 
Furthermore, if $\gamma$ is rational then so is each
rectangular union.
\end{lemma}

\begin{proof}
Consider the rectilinear curve $\gamma$ in $\R^2$. If it bounds an immersed disk, then $\gamma$ can be oriented so that the winding number around any point in the plane is non-negative.
By rectilinearity,
we can divide the bounded components of $\R^2 \smallsetminus \gamma$ into rectangles. Furthermore, if $\gamma$ is rational,
these rectangles can be chosen all to be rational.
Let $\sR$ be the collection of such closed rectangles with multiplicity corresponding to the winding number. 

Each immersed disk bounded by $\gamma$ can be assembled by identifying boundary edges of rectangles in $\sR$. In particular by Proposition \ref{prop:finite},
there are only finitely many immersed disks with boundary $\gamma$. Furthermore, from this construction we see that each such disk is a rectangular union. If $\gamma$ was rational,
then the disk is a rational rectangular union.
\end{proof}

\begin{corollary}\label{cor:smallest}
\note{Used in Theorem \ref{thm:nested}.}
Let $P$ be a planar surface and let $K \in \PC(P)$ be a closed rectangular union. Then, there is a smallest $D \in \cdisk(P)$ so that $K \subset D$. Furthermore,
$D$ is a rectangular union. 
\end{corollary}
\begin{proof}
By definition of rectangular union, $\partial K$ is a union of disjoint simple closed curves, each of which bounds a disk in $P$. Let $A$ be the unique unbounded component of $P \smallsetminus K$. Then $A$ can only touch one boundary component of $P$, so $A$ is homeomorphic to an annulus, and $\partial A \subset P$ consists of this one component.
The developed image
$\dev(\partial A)$ is a rectilinear curve. So, it bounds
a rectangular union in $\cdisk(P)$ by the lemma above.
\end{proof}

If $K \in \PC(P)$ is a closed rectangular union,
then we call the disk $D$ provided by the corollary the {\em smallest closed disk containing $K$}. Similarly, any open set $U$ in a planar surface which is homeomorphic to a finitely punctured disk is contained in a {\em smallest open disk} obtained by filling in the compact components of the compliment.

\begin{corollary}
\label{cor:finite open disks}
\note{used in the proof of Theorem \ref{thm:open}.}
Let $P$ be a planar surface and $U \subset P$ be an open subset containing the basepoint and homeomorphic to a finitely punctured disk. If $Q \in \tM$ and there is an embedding $e:U \emb Q$, let $D(Q) \subset Q$ denote the smallest open disk containing $e(U)$. Then the set 
$$\{D(Q):~\text{$Q \in \tM$ and $U \emb Q$}\}$$
contains only finitely many different isomorphism classes. 
\end{corollary}
\begin{proof}
Orient the boundary components of $U$ so that traveling around $\partial U$ leaves $U$ on the left. For each boundary component choose a smooth simple closed curve in $U$ homotopic to the boundary component with the same orientation. This allows us to distinguish the boundary component which does not bound a compact subset of $P$, namely the one whose associated curve has turning number $1$. (The other curves have turning number $-1$ since they bound a disk with the wrong orientation.)

Now we can choose a closed curve $\tilde \gamma$ in $U$ which is homotopic to the distinguished boundary component so that $\gamma=\Dev|_P \circ \tilde \gamma$ is a rectilinear curve. Let $\sV \subset \tM$ denote the collection isomorphism classes of planar surfaces bounded by lifts of $\gamma$. The set $\sV$ finite by Lemma \ref{lem:2}.

Now suppose that $R$ is a planar surface and there is an embedding $e:U \emb R$. Let $V$ be the open disk bounded by $e(\tilde \gamma)$. 
The isomorphism class of $V$ lies in $\sV$.
The set $e(U) \cup V$ is the smallest open disk containing $e(U)$. 
Furthermore this disk is uniquely determined by the isomorphism class $V$.
To see this suppose that $R'$ is another planar surface,
$e': U \emb R'$ is an embedding, $V'$ is the open disk bounded by $e(\tilde \gamma)$ and $V$ and $V'$ are isomorphic. Let $f:V \to V'$ denote this isomorphism. Define the map 
$$g:e(U) \cup V \to e'(U) \cup V'; \quad x \mapsto \begin{cases}
e' \circ e^{-1}(x) & \text{if $x \in e(U)$}\\
f(x) & \text{if $x \in V$}.\end{cases}$$
This map is ambiguous on $V \cap e(U)$, but this set is a path-connected open set and the maps $e' \circ e^{-1}$ and $f$ restricted to this set are both embeddings so they agree. Furthermore, $g$ is an immersion because it is an immersion locally. We similarly get a map from $e'(U) \cup V' \to e(U) \cup V$, showing that these two sets are isomorphic. This proves that the isomorphism class of the smallest open disk containing an embedded image of $U$ is uniquely determined by $V$ as defined above and taken from the finite set $\sV$.
\end{proof}

\subsection{Constructing rectangular unions}

\begin{theorem}
\label{thm:nested}
\note{Definitely used.}
Let $P$ be a planar surface, and let $K_1, U \in \PC(P)$ with $K_1$ compact, $U$ open and $K_1 \subset U$.
Then, there is a closed rational rectangular union 
$K_2 \in \PC(P)$ so that $K_1 \subset K_2^\circ$ and
$K_2 \subset U$. Furthermore, if $U$ is homeomorphic to an open disk then we can arrange that $K_2 \in \cdisk(P)$.
\end{theorem}
\begin{proof}
We will deal with the last sentence latter in the proof, for now we will produce a $K_2 \in \PC(P)$.
For every $p \in K_1$, choose a closed rational rectangle $R_p$ 
so that $p \in R_p \subset U$.
Then $\{R_p^\circ~:~p \in K_1\}$ is an open cover of $K_1$. 
By compactness, there is a finite subcover. Let $\{R_1,\ldots, R_n\}$ be the corresponding collection of closed rectangles. Let $K_2 = \bigcup_{i=1}^n R_i$. This set is path connected and contains the  basepoint because $K_1$ does. Also by construction, we have $K_1 \subset K_2^\circ$ and
$K_2 \subset U$. 

It may not be true that $K_2$ is a rectangular union because 
$\partial K_2$ could fail to be bounded by disjoint curves. This can only happen if some rectangles share a common vertex but are situated diagonally across from each other as depicted below:
\begin{center}
\includegraphics{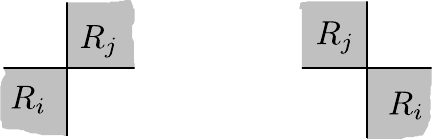}
\end{center}
We can fix this problem, by adding a rectangle centered at the common vertex which is small enough to be contained in $U$ and only intersect the edges of rectangles in the set $\{R_1, \ldots, R_n\}$ in edges which contain
the common vertex.

Now suppose that $U$ is a topological disk. Replacing the $K_2$ constructed above by the smallest closed disk containing $K_2$ gives the last statement. See Corollary \ref{cor:smallest}.
\end{proof}

\subsection{\texorpdfstring{Open sets in $\tM$}{Open sets in M}}
\label{sect:Open sets in M}
In this subsection we prove Theorem \ref{thm:open} namely that sets of the form $\tM_{\imm}(K)$, $\tM_{\emb}(K)$, $\tM_{\not \imm}(U)$ and $\tM_{\not \emb}(U)$ are open in $\tM$. 

\begin{proof}[Proof of Theorem \ref{thm:open}]
We prove Theorem \ref{thm:open} using the definition of the immersive topology. That is, we only assume set of the form
$\tM_{\imm}(K)$ and $\tM_{\not \emb}(U)$ are open when $K$
and $U$ are closed and open topological disks, respectively.

Let $P$ be a planar surface and $K \in \PC(P)$ be compact. We will show that $\tM_{\imm}(K)$ is open. Choose $Q \in \tM_{\imm}(K)$. Then by definition, there is an immersion 
$\iota:K \imm Q$. Let $K_1=\iota(K)$. By choosing a closed disk family in $Q$, we can find a $K_3 \in \cdisk(Q)$ so that
$K_1 \subset K_3^\circ$. Then Theorem \ref{thm:nested}
guarantees that there is a closed rational rectangular union $K_2 \in \cdisk(Q)$ so that $K_1 \subset K_2^\circ$ and $K_2 \subset K_3^\circ$. Since $\tM_{\imm}(K_2)$ is open, it suffices to prove that 
$\tM_{\imm}(K_2) \subset \tM_{\imm}(K)$. Let $R \in \tM_{\imm}(K_2)$. Then, there is an immersion $j:K_2 \imm R$. The composition $j \circ \iota:K \imm R$ is the immersion needed
to prove that $R \in \tM_{\imm}(K)$. 

Let $P$ be a planar surface and $U \in \PC(P)$ be open. We will show that $\tM_{\not \emb}(U)$ is open. Choose $Q \in \tM_{\not \emb}(U)$. Choose an exhaustion of $U$ by an increasing sequence of compact sets $D_n \in \PC(P)$ whose interiors $D_n^\circ$ lie in $\PC(P)$ so that each $D_n \subset D_{n+1}^\circ$.  
By Proposition \ref{prop:supremum}, if each $D_n^\circ$ embeds in $Q$, then $U$ would embed in $Q$. 
So there is an $a>0$ so that $D_a \not \emb Q$. By Theorem \ref{thm:nested},
there is a closed rectangular union $K_2$ so that
$D_a \subset K_2^\circ$ and $K_2 \subset D_{a+1}^\circ$.
Because $D_a^\circ \subset K_2^\circ \subset U$, we have
$$Q \in \tM_{\not \emb}(D_a^\circ) \subset \tM_{\not \emb}(K_2^\circ) \subset \tM_{\not \emb}(U).$$
It suffices to show that $\tM_{\not \emb}(K_2^\circ)$ is open.
Let $\sD$ be the collection of all isomorphism classes
of smallest open disks containing embedded images of $K_2^\circ$. 
This set is finite by Proposition \ref{prop:open disk} together with Corollary 
\ref{cor:finite open disks}. If
$K_2^\circ$ embeds in a planar surface $R$, then there
is an element $D \in \sD$ which also embeds. It follows that
$$\tM_{\not \emb}(K_2^\circ)=\bigcap_{D \in \sD} \tM_{\not \emb}(D),$$
which is open by definition of the topology.

Let $P$ be a planar surface and $U \in \PC(P)$ be open. We will now show that $\tM_{\not \imm}(U)$ is open.
Choose $Q \in \tM_{\not \imm}(U)$. By the same reasoning as above, we can find an closed rectangular union $K_2 \in \cdisk(P)$
so that $K_2^\circ \subset U$ and $K_2^\circ \not \imm Q$. 
Furthermore, $\tM_{\not \imm}(K_2^\circ) \subset \tM_{\not \imm}(U).$
Let $\sV$ be the collection of all immersed images of $K_2^\circ$.
The set $\sV$ is finite by Corollary \ref{cor:finite immersed}. From the above paragraph, we know that $\tM_{\not \emb}(V)$ is open for every $V \in \sV$. Thus,
$$\tM_{\not \imm}(K_2^\circ)=\bigcap_{V \in \sV} \tM_{\not \emb}(V)$$
is open.

Finally, let $P$ be a planar surface and $K \in \PC(P)$ be compact. We will show that $\tM_{\emb}(K)$ is open. 
Choose
any $Q \in \tM_{\emb}(K)$. Then there is an embedding $e:K \emb Q$. Let $K_1=e(K)$. Choose $K_3 \in \cdisk(Q)$ so that
$K_1 \subset K_3^\circ$. Then we can find a $K_2 \in \cdisk(Q)$ which is a closed rectangular union
and satisfies $K_1 \subset K_2^\circ$ and $K_2 \subset K_3^\circ$. 
Let $\sL$ be the collection of all immersed images of $K_2^\circ$ up to isomorphism. This collection is finite by Corollary \ref{cor:finite immersed}. Let $\sL_{0}$ be $\sL$ with the equivalence class of $K_2^\circ$ itself removed. Then if $L \in \sL_0$, the immersion $K_2^\circ \imm L$ is not an embedding.
Let $R$ be a planar surface. Suppose that there is an immersion $\iota:K_2 \imm R$
and that for each $L \in \sL_0$, we have $L \not \emb R$.
Note that by restriction of $\iota$, we have $K_2^\circ \imm R$. Then by definition of $\sL_0$, our immersion 
$K_2^\circ \imm R$ must actually be an embedding. So, by restriction, $\iota|_{K_1}:K_1 \emb R$, and by composition 
$\iota \circ e:K \emb R$. It follows that 
$$Q \in \tM_{\imm}(K_2) \cap \bigcap_{L \in \sL_0} \tM_{\not \emb}(L) \subset \tM_{\emb}(K).$$
This provides an open neighborhood of $Q$ contained in $\tM_{\emb}(K)$. 
\end{proof}

\subsection{Countable subbases}

\begin{proposition}[Explicit second-countability]
\label{prop:second countability}
The subsets of the following two forms give a countable subbasis for the topology on $\tM$:
\begin{itemize}
\item Sets of the form $\tM_{\imm}(K)$ where $K \in PC$
is an isomorphism class of a closed rational rectangular union.
\item Sets of the form $\tM_{\not \emb}(U)$ where $U\in PC$
is an isomorphism class of an open rational rectangular union.
\end{itemize}
\end{proposition}
\begin{proof}
Countability follows from Corollary \ref{cor:countably many classes}.

The sets listed are clearly open by Theorem \ref{thm:open}.
We must prove that they form a subbasis for the topology. 
We will show that the sets in the subbasis used to define
the immersive topology in Definition \ref{def:immersive topology on M}
are open in the topology $\sT'$ generated by the sets listed in this proposition. 

Let $P$ be a planar surface and let $K \in \cdisk(P)$. We will show $\tM_{\imm}(K)$ is open in $\sT'$. Let $Q \in \tM_{\imm}(K)$. Then there is an immersion $\iota:K \imm Q$. Let $K_1=\iota(K)$. Using a closed disk family, we can find a $K_3 \in \cdisk(Q)$ so that $K_1 \subset K_3^\circ$. 
Then, Theorem \ref{thm:nested}
guarantees that there is a closed rational rectangular union $K_2 \in \cdisk(Q)$ so that $K_1 \subset K_2^\circ$ and $K_2 \subset K_3^\circ$. Observe that 
$Q \in \tM_{\imm}(K_2)$ and $\tM_{\imm}(K_2) \subset \tM_{\imm}(K).$
It follows that $\tM_{\imm}(K)$ is open in $\sT'$.

Let $P$ be a planar surface and let $U \in \cdisk(P)$. We will show that $\tM_{\not \emb}(U)$ is open in $\sT'$. Let $Q \in \tM_{\not \emb}(U)$. Then $U \not \emb Q$. Since $U$ is an open disk, we can think of it as a planar surface. Choose a closed disk family $\{K_t~:~t>0\}$ for $U$. By Proposition \ref{prop:supremum}, there is an $a>0$ so that $K_a^\circ \not \emb Q$. Let $K \in \cdisk(P)$ be a rectangular union
satisfying $K_a \subset K^\circ$ and $K \subset K_{a+1}^\circ$. We have $Q \in \tM_{\not \emb}(K^\circ)$ and
$\tM_{\not \emb}(K^\circ) \subset \tM_{\not \emb}(U)$. So,
$\tM_{\not \emb}(U)$ is open in $\sT'$.
\end{proof}

\begin{proposition}[Explicit second-countability of $\tE$]
\label{prop:second countability2}
A countable subbasis for the immersion topology on 
$\tE$ is given by the union of the
collection of preimages under $\tpi$ of the subbasis provided by Proposition 
\ref{prop:second countability} together with the collection sets of the form 
$\tE_\imm(K,U)$ where $K \in PC$
is an isomorphism class of a closed rational rectangular union and $U \subset K^\circ$ is an open rational rectangle.
\end{proposition}
\begin{proof}
The potential subbasis described is clearly a countable collection of open sets. We must show that it generates the topology. By Proposition \ref{prop:second countability}, the map $\tpi$ is continuous in the generated topology. To conclude the proof, we must show that $\tE_\imm(D,V)$ is open for an arbitrary $D \in \cdisk(P)$ and arbitrary $U \subset D^\circ$ open.
Choose a $(Q,q) \in \tE_\imm(D,V)$. Then there is an
immersion $\iota:D \imm Q$ and $q \in \iota(V)$. 
By taking a closed disk family in $Q$ and applying Theorem \ref{thm:nested}, we can produce a closed rational rectangular 
union $K \in \cdisk(Q)$ so that $\iota(D) \subset K^\circ$.
Also since $q \in \iota(V)$ and $\iota(V)$ is open, we can find an open rational rectangle $U$ so that $q \in U \subset \iota(V)$. Then, $(Q,q) \in \tE_\imm(K,U)$. We also claim that
$\tE_\imm(K,U) \subset \tE_\imm(D,V)$. Suppose $(R,r) \in \tE_\imm(K,U)$. Then, there is an immersion $j:K \imm R$ and
$r \in j(U)$. By composition, we have an immersion
$j \circ \iota:D \imm R$. Furthermore, since $U \subset \iota(V)$, we have $r \in j(U) \subset j \circ \iota(V)$. 
\end{proof}

\section{Sequences}
\label{sect:sequences}
We have showed that the topologies on $\tM$ and $\tE$ are second-countable. We recall that a map $f:X \to Y$ between second-countable spaces is continuous if and only if it is sequentially continuous; see \cite[Theorem 30.1]{Munkres}. Therefore, we will begin to use sequences to verify the continuity of maps. We use this section to describe criteria for convergence
and consequences of convergence in $\tM$ and $\tE$.

\begin{proposition}[Criterion for convergence in $\tM$.]
\label{prop:conv}
Let $P \in \tM$ be a planar surface
and let $\langle P_n \rangle_{n \in \N}$ be a sequence of
planar surfaces. 
Suppose the following two statements hold:
\begin{enumerate}
\item[(A)] If $D \in \cdisk(P)$, then $D \imm P_n$ for all but finitely many $n$. 
\item[(B)] If $Q$ is a planar surface, and
$Q \emb P_n$ for infinitely many $n$, then $Q \imm P$.
\end{enumerate}
Then, $\langle P_n \rangle$ converges to $P$ in the immersive topology on $\tM$.
\end{proposition}

\begin{example}
\label{ex: convergence}
Let $P_n$ be the universal cover of $\C_n \smallsetminus R_n$, where $R_n$ denotes the set $n$-th roots of unity. Then $P_n$ naturally has a translation structure obtained by pullback under the covering map. We will observe that $P_n$ tends to the open unit disk $P$ in the immersive topology. We will check statements of Proposition \ref{prop:conv}. 
First of all observe that there are embeddings $\epsilon_n:P \imm P_n$ for all $n$. By restricting this embedding we see any closed disk in $P$ immerses in every $P_n$. To see (B) suppose $Q$ is a planar surface which embeds in infinitely many $P_n$. Then,
$Q$ can be viewed as a subset of some $P_n$. It suffices to show that $Q$ is entirely contained in the image $\epsilon_n(P)$, since in this case
$\epsilon_n^{-1}$ restricted to $Q$ is an embedding into $P$. Suppose $Q$ is not contained in $\epsilon_n(P)$, we see $Q$ intersects the boundary of $\epsilon_n(P)$. Let $\ell$ be the arc length of this interval of intersection viewed as a subset of the unit circle. Choose $M \in \N$ so that for $m>M$, the roots of unity separate the unit circle into intervals of length less than $\ell$. Then for $m>M$ some root of unity lies in the immersed image of $Q$ in the plane. Therefore, $Q \not \imm \C_m \smallsetminus R_m$ and thus $Q \not \imm P_m$ for $m>M$. This contradicts the assumption that $Q$ embeds in infinitely many $P_m$. 
\end{example}

The following is a direct consequence of Theorem \ref{thm:open}, so we will not include a proof.

\begin{corollary}[Necessary conditions for convergence in $\tM$.]
\label{cor:necessary M}
Suppose $\langle P_n \in \tM \rangle_{n \in \N}$ is a sequence of
planar surfaces converging to $P \in \tM$. Then, the following two statements are satisfied:
\begin{enumerate}
\item[(A')] Suppose $K \in \PC$ is compact. 
Then $K \imm P$ implies $K \imm P_n$ for $n$ sufficiently large,
And, $K \emb P$ implies $K \emb P_n$ for $n$ sufficiently large.
\item[(B')] Suppose $U \in \PC$ is open. Then $U \imm P_n$ for infinitely many $n$
implies $U \imm P$. And, $U \emb P_n$ for infinitely many $n$ implies $U \emb P$. 
\end{enumerate}
\end{corollary}

Recall the definition of $\tE_\imm(K,U)$ in \eqref{eq:E+}.
The following provides a criterion for convergence in $\tE$.
\begin{proposition}[Convergence in $\tE$]
\label{prop:convergence 2}
Suppose $P_n \in \tM$ is a sequence in converging to $P \in \tM$ in the immersive topology.
Let $d_n$ denote the Euclidean path metric on $P$. 
Let $p_n \in P_n$ and $p \in P$ be a choice of points on these surfaces.
Then the following are equivalent:
\begin{enumerate}
\item $(P_n, p_n)$ converges to $(P,p)$ in the immersive topology on $\tE$.
\item There is a compact set $K \in \PC(P)$ which contains $p$ and an $N$ so that there is an immersion
$\iota_n:K \imm P_n$ defined for $n>N$ so that $d_{n}\big(p_n, \iota_n(p)\big) \to 0$ as $n \to \infty$.
\item For every compact set $K \subset \PC(P)$ containing $p$, there is an $N$ and an embedding $e_n:K \emb P_n$
defined for $n>N$ so that $d_{n}\big(p_n, e_n(p)\big) \to 0$ as $n \to \infty$.
\end{enumerate}
\end{proposition}

\begin{example}[Example \ref{ex: convergence} continued]
\label{ex: convergence2}
Let $P_n$, $P$ and $\epsilon_n$ be as in Example \ref{ex: convergence}. Choose $p_n \in P_n$ for all $n$. 
We claim that $(P_n, p_n)$ converges if there is an $N$ such that $p_n \in \epsilon_n(P_n)$ for $n>N$ and
$\epsilon_n^{-1}(p_n)$ converges to some point $p \in P$. In this case $(P_n,p_n) \to (P,p)$. Considering the second statement of Proposition \ref{prop:convergence 2}, it suffices to choose $K$ to be a closed ball about the origin which contains $p$ in its interior, and consider the second statement for a sequence of values of $U$ where $U$ is an open metric ball about $p$ with radii tending to zero. (Using the definition of the immersive topology, it can be observed that in fact we have given a characterization of convergence of $(P_n,p_n)$ to $(P,p)$ for some $p \in P$.)
\end{example}

We note the following consequence.
\begin{corollary}
\label{cor:approximate points}
Suppose that the sequence $P_n \in \tM$ converges to $P \in \tM$. Then for any $p \in P$, there is a an $N$ and a sequence
$p_n \in P_n$ defined for $n>N$ with $\Dev(P_n,p_n)=\Dev(P,p)$ so that $(P_n,p_n)$ converges to $(P,p)$ in $\tE$. 
\end{corollary}

\subsection{Proofs}
\label{sect: proofs of convergence criteria}
The following is the proof of our convergence criterion for $\tM$:

\begin{proof}[Proof of Proposition \ref{prop:conv}]
We will suppose statements (A) and (B) of the Proposition are satisfied and prove that for any closed disk $K$, $P \in \tM_{\imm}(K)$ implies that $P_n \in \tM_{\imm}(K)$
for all but finitely many $n$, and that for any open disk $U$, $P \in \tM_{\not \emb}(U)$ implies 
$P_n \in \tM_{\not \emb}(U)$
for all but finitely many $n$.

Let $K$ be a closed disk in a planar surface $Q$ and suppose that $P \in \tM_{\imm}(K)$.
Then there is an immersion $\iota:K \imm P$. By taking a closed disk family in $P$, we can find
a closed disk $D$ in that family so that $\iota(K) \subset D$. From (A), we get immersions $\iota_n:D \imm P_n$ for
all but finitely many $n$. Whenever $\iota_n$ is defined, the composition
$\iota_n \circ \iota$ is an immersion $K \imm P_n$. This proves $P_n \in \tM_{\imm}(K)$
for all but finitely many $n$.

Now suppose that $U$ is an open disk in a planar surface $Q$ and that $P \in \tM_{\not \emb}(U)$.
Then, $U \not \emb P$. Suppose it is not true that $P_n \in \tM_{\not \emb}(U)$
for all but finitely many $n$. Then there is an increasing sequence of integers $\langle n_k\rangle$
and embeddings $e_k:U \emb P_{n_k}$. Since $U$ is an open disk, it is isomorphic to a planar surface which we abuse notation by also denoting $U$.
So, by (B) applied to $U$, we know that there is an immersion $\iota:U \imm P$. Suppose it is not an embedding.
Then there are points $u,v \in U$ so that $\iota(u)=\iota(v)$. Let $K \subset U$ be a closed disk in $Q$ containing both $u$ and $v$. Then $\iota(K)$ is compact. As in the prior paragraph, we can find a closed disk $D$
so that $\iota(K) \subset D$. Then, by (A) we get immersions $\iota_n:D \imm P_n$ for all but finitely many $n$.
In particular, for sufficiently large $k$, we have $\iota_{n_k} \circ \iota|_K:K \imm P_{n_k}$.
But we also get such an immersion as a restriction of an embedding, $e_k|_K:K \emb P_{n_k}$.
Since immersions are unique, it follows that $e_k|_K=\iota_{n_k} \circ \iota|_K$.
Therefore, $\iota|_K$ must be injective, which contradicts the statement above that $\iota(u)=\iota(v)$.\end{proof}

We now prove our equivalences for convergent sequences in $\tE$:

\begin{proof}[Proof of Proposition \ref{prop:convergence 2}]
As in the statement of the proposition, let $(P_n,p_n) \in \tE$ be a sequence,
and consider convergence to $(P,p) \in \tE$. By hypothesis, we know $P_n \to P$ in $\tM$, and we have a $K \in \cdisk(P)$ so that $p \in K^\circ$ and for any open set $U \subset K^\circ$, $(P_n,p_n) \in \tE_\imm(K,U)$ for all but finitely many $n$. We wish to show $(P_n,p_n) \to (P,p)$ in $\tE$. 

By definition, we have $(P_n,p_n) \to (P,p)$ if $P_n \to P$ (which it does) and
whenever $Q$ be a planar surface, $D \in \cdisk(Q)$, $V \subset D^\circ$ is open
and $(P,p) \in \tE_\imm(D,V)$, then there is an $N$ so that $n>N$ implies $(P_n,p_n) \in \tE_\imm(D,V)$. 

Fix $Q$, $D \subset \cdisk(Q)$, and an open $V \subset D^\circ$ so that $(P,p) \in \tE_\imm(D,V)$. Then by definition, there is an immersion $\iota:D \imm P$
and $p \in \iota(V)$. Select another closed disk $L \in P$ so that
$K \subset L$ and $\iota(D) \subset L$. Since $P_n \to P$, we know that there is an $N_1$ so that $n>N_1$ implies $\exists \iota_n:L \imm P_n$. Let $U=K^\circ \cap \iota(V)$. Then, $p \in U$ and we conclude by hypothesis that $(P_n,p_n) \in \tE_\imm(K,U)$ for $n>N_2$. Now let $n> \max(N_1, N_2)$. We can see
$(P_n,p_n) \in \tE_\imm(D,V)$, because
\begin{enumerate}
\item $\iota_n|_{\iota(D)} \circ \iota$ is an immersion of $D$ into $P_n$.
\item We know $p_n \in \iota_n(U)$ and $U \subset \iota(V)$, thus
$p_n \in \iota_n \circ \iota(V)$.
\end{enumerate}
\end{proof}

\begin{proof}[Proof of Proposition \ref{prop:convergence 2}]
First we show (1) implies (3). Suppose $(P_n,p_n) \to (P,p)$. Fix a $K \in \PC(P)$ containing $p$.
Since $P \to P$ and $K \emb P$,
we know that for $n$ sufficiently large there is an embedding $e_n:K \emb P_n$; see Corollary \ref{cor:necessary M}.
We claim $d_{n}\big(p_n, e_n(p)\big) \to 0$. Choose an $\epsilon>0$. Choose a compact disk $D \subset P$
so that $K \subset D$ and so that $p \in D^\circ$. Then we can choose an open ball $B$ about $p$ of radius less than $\epsilon$
so that $U \subset D^\circ$. Since $(P,p) \in \tE_\imm(D,U)$, we know that there is an $N$ so that $(P_n,p_n) \in \tE_\imm(D,U)$ for $n>N$. But then there is an immersion $\iota_n:D \imm P_n$ so that $p_n \in \iota_n(U)$. But since $U$ is an small ball about 
$p$, we know that this implies that $d_n\big(p_n,\iota_n(p)\big)<\epsilon$. Finally, since immersions
are unique, we know $\iota_n=e_n$ on $K$, so the same holds with $e_n(p)$ replacing $\iota_n(p)$.

Clearly (3) implies (2). We will finish the proof by showing that (2) implies (1). Let $K$, $N$ and $\iota_n$ be as in (2). By definition of the immersive topology, we need to show that for every closed disk $D \in \PC$ and every open $U \subset D^\circ$ so that $(P,p) \in \tE_\imm(D,U)$, we have $(P_n,p_n) \in \tE_\imm(D,U)$ for $n$ sufficiently large. 
Fix $D$ and $U$ and suppose $(P,p) \in \tE_\imm(D,U)$ so that there is an immersion $j:D \imm P$ and $p \in j(U)$.
Since $j(U)$ is open, there is an $\epsilon>0$ so that the open $\epsilon$ ball about $p$ is contained in $j(U)$.
Choose a closed disk $K' \in P$ so large that it contains both $K$ and $j(D)$. Then since $P_n \to P$,
we know that there is an $N'>N$ and immersions $\iota_n':K' \to P_n$ for $n>N'$. Observe that the image
$\iota_n'\circ j(U)$ contains an $\epsilon$-ball about $\iota_n'(p)$ because $\iota_n'$ is a local translation
and $j(U)$ contains such a ball about $p$. By uniqueness of immersions
we know $\iota_n'|_K=\iota_n$, and in particular $\iota_n(p)=\iota_n'(p)$.
By hypothesis, there is an $N''>N'$ so that for $n>N''$ 
we have $d_{n}\big(p_n, \iota_n(p)\big)<\epsilon$
which then implies $p_n \in \iota_n'\circ j(U)$ by prior remarks. 
Observe that when $n>N''$, we have that 
$\iota_n' \circ j|_D$ is an immersion of $D$ into $P_n$ and $p_n \in \iota_n' \circ j(U)$
so that $n>N''$ implies $(P_n,p_n) \in \tE_\imm(D,U)$ as desired.
\end{proof}

\begin{proof}[Proof of Corollary \ref{cor:approximate points}]
Let $K \subset P$ be a closed disk containing $p$ in its interior. Let $\iota_n:K \to P_n$ be immersions guaranteed
to exist for $n$ sufficiently large because $P_n \to P$ and $\tM_{\imm}(K)$ is open. Let $p_n=\iota_n(p)$.
Then we see that by definition that $(P_n,p_n) \in \tE_\imm(K,U)$ whenever $\iota_n$ is defined and when $p$ lies in the open subset $U \subset K^\circ$. The developing map comment is true because $\Dev$ is invariant under immersions.
\end{proof}

\section{Continuity of immersions}
\label{sect:continuity of immersions}
The following explains that immersions and embeddings are jointly continuous in choice of the domain and range, and that the natural domains for these maps are closed.

\begin{theorem}
\label{thm:Continuity of immersions}
The sets 
$$\{(P,Q) \in \tM \times \tM:~ P \imm Q\} \quad \text{and} \quad 
\{(P,Q) \in \tM \times \tM:~ P \emb Q\}$$
are closed in $\tM \times \tM$. 
The function 
$$I:\Big\{\big((P,p), Q\big) \in \tE \times \tM:~P \imm Q\Big\} \to \tE$$
which sends $\big((P,p), Q\big)$ to the image of $p$ under the immersion $P \imm Q$
has a closed domain and is continuous.
\end{theorem}

\begin{proof}[Proof of Theorem \ref{thm:Continuity of immersions}]
We begin with dealing with the first sentence in the case of embeddings. 
Suppose $\{P_n\}$ and $\{Q_n\}$ are sequences of planar surfaces
each of which converges to $P$ and $Q$ respectively
and $P_n \emb Q_n$ for all $n$. We need to show $P \emb Q$. 
It suffices to show that any closed disk $D$ in $P$ embeds in $Q$ by Proposition \ref{prop:supremum}. Let $K \subset P$ be a closed disk containing $D$ in its interior. Observe that $K$ embeds in all but finitely many $P_n$ since $\tM_{\emb}(K)$ is open. 
Then by composition with $P_n \emb Q_n$, $K \emb Q_n$ for infinitely many $n$. Since $\tM_{\not \emb}(K^\circ)$ is open, we see that $K^\circ$ embeds in $Q$, and thus so does $D$. 

For the case of immersions we need rectangular unions. 
Suppose $P_n \to P$, $Q_n \to Q$ and each $P_n \imm Q_n$. We will show $P \imm Q$.
Again it suffices to show that every open disk $D$ with compact closure in $P$ immerses in $Q$. Let $D \subset P$ be such a disk, and let $V \subset P$ be an open rectangular union containing $D$ which is a topological disk with compact closure $\bar V$. Since $\tM_{\emb}( \bar V)$ is open by Theorem \ref{thm:open}, there are embeddings $e_n:\bar V \emb P_n$ for $n$ sufficiently large. By composing with immersions $P_n \imm Q_n$, we see 
$V \imm Q_n$ for all but finitely many $n$. By Corollary \ref{cor:finite immersed}, there are only finitely many such images of $V$ up to isomorphism. 
Let $W$ be an isomorphism class which appears infinitely many times. By Theorem \ref{thm:open},
the set $\tM_{\not \emb}(W)$ is open, so it must be that $W \emb Q$. Note that since $V \imm W$ we have $D \imm W$. Thus $D \imm Q$, which completes the proof that $P \imm Q$. 

We note that the domain of $I$ is the preimage of $\{(P,Q):~P \imm Q\}$
under $\tpi \times \text{id}$ where $\tpi: \tE \to \tM$ is the projection.
Thus the domain of $I$ is closed.

Now let $\big((P_n,p_n), Q_n\big)$ be a sequence in the domain of $I$ which converges to 
$\big((P,p), Q\big) \in \tE \times \tM$. So there are immersions $\iota_n:P_n \imm Q_n$,
and there is an immersion $\iota:P \imm Q$ by remarks above. We need to show that
$\big(Q_n, \iota_n(p_n)\big) \to (Q,\iota(p)\big)$. Choose $K \subset P$ compact and containing $p$ in its interior.
It follows that there is an $\epsilon>0$ so that $K^\circ$ contains an $\epsilon$-ball about $p$.
By Proposition \ref{prop:convergence 2} we know that for $n$ sufficiently large there is an embedding $e_n:K \emb P_n$
and $d_{P_n}\big(p_n, e_n(p)\big) \to 0$ as $n \to \infty$. When this distance is less than $\epsilon$, we know that
\begin{equation}
\label{eq:distance equal}
d_{P_n}\big(e_n^{-1}(p_n), p\big)=d_{P_n}\big(p_n, e_n(p)\big),
\end{equation}
so the quantity at left also tends to zero. Let $K'=\iota(K)$ which is compact. Since we know $Q_n \to Q$, for $n$ sufficiently large there is an embedding $f_n: \iota(K) \emb Q_n$. Observe that $f_n \circ \iota:K \imm Q_n$. Since immersions non-strictly contract distances, by applying this immersion we see
$$d_{Q_n}\big(f_n \circ \iota \circ e_n^{-1}(p_n), f_n \circ \iota(p)\big) \to 0 \quad \text{as $n \to \infty$}.$$
Now observe that both $f_n \circ \iota \circ e_n^{-1}$ and $\iota_n|_{\iota(K)}$ give immersions of $\iota(K) \imm Q_n$
so they are equal. It therefore follows that $d_{Q_n}\big(\iota_n(p_n), f_n \circ \iota(p)\big) \to 0$ as $n \to \infty$.
Then by Proposition \ref{prop:convergence 2} it follows that $\big(Q_n,\iota_n(p_n)\big) \to \big(Q, \iota(p)\big)$ as desired.
\end{proof}

\section{Compact subsets and Metrizability}
\label{sect:limits and compactness}

In this section, we prove the following theorem and establish consequences such as metrizability of $\tM$ and $\tE$.

\begin{theorem}
\label{thm:compactness}
Let $P$ be a planar surface. The set of surfaces
$$\tM \smallsetminus \tM_{\not \imm}(P)=\{Q \in \tM~:~P \imm Q\}$$
is compact.
\end{theorem}

It follows the only way a sequence $\langle P_n \rangle$ of planar surfaces can leave every compact set of $\tM$ is if the radius of largest open Euclidean metric ball we can immerse in $P_n$ centered at the basepoint tends to zero as $n \to \infty$.

We establish two important consequences of this result:

\begin{corollary}[Local compactness]
\label{cor:local compactness}
Both $\tM$ and $\tE$ are locally compact in the sense that every point in these spaces has a compact neighborhood.
\end{corollary}

\begin{corollary}[Metrizability]
\label{cor:metrizability}
The spaces $\tM$ and $\tE$ are metrizable.
\end{corollary}

\subsection{Direct limits}

For the proof of Theorem \ref{thm:compactness}, we need to know that an $\imm$-increasing sequence
converges:
\begin{proposition}[Direct limit]
\label{prop:direct limit}
Suppose $\langle P_n \in \tM \rangle_{n \geq 1}$ is a sequence satisfying
$$P_1 \imm P_2 \imm P_3 \imm \ldots.$$
Then, the sequence converges to $\Fuse \{P_n\}$.
\end{proposition}

\begin{proof}
Let $\langle P_n\rangle_{n \geq 1}$ be a sequence of planar surfaces as stated in the proposition. Then
for each $m,n$ with $m \leq n$, there is an immersion $\iota_{m,n}:P_m \imm P_n$. Let $P_\infty=\Fuse \{P_n\}$. 
By the Fusion Theorem, there are immersions 
$j_n:P_n \to P_\infty$.
To prove that the sequence $\langle P_n\rangle$ converges to $P_\infty$, we apply the convergence criterion of Proposition \ref{prop:conv}.

Let $K \in \cdisk(P_\infty)$.
The set $K$ is compact, and we will apply a compactness argument
to say that there is an $M$ and a lift $\tilde K \subset P_M$ containing the basepoint of $P_M$
so that $j_M|_{\tilde K}$
is a homeomorphism from $\tilde K$ onto $K$ which respects the basepoints.
Then, the inverse of this restriction
$(j_M|_{\tilde K})^{-1}$ is the needed immersion of $K$ into $P_M$. We can then immerse
$K$ into all $P_N$ with $N>M$ by composing with $\iota_{M,N}$. This will prove the first
statement needed from Proposition \ref{prop:conv}.

We will now construct $\tilde K$. Observe that by definition of the fusion,
$P_\infty=\bigcup_n j_n(P_n)$. Otherwise, the union would be a smaller trivial surface with the properties of the fusion, violating the definition of $P_\infty$ as the fusion. See Theorem \ref{thm:generalized fusion}. 

Therefore, for each $x \in K$, there is an
$n(x) \geq 1$ and a $p(x) \in P_{n(x)}$ so that $j_{n(x)}\big(p(x)\big)=x$. 
For the basepoint of $P_\infty$ in $K$, we take $n(x)=1$ and $p(x)$ to be the basepoint of $P_1$. 
For each $x$, let $B_x \subset P_{n(x)}$ be an open metric ball about $p(x)$ so that $\dev|_{B_x}$ is a homeomorphism onto a ball in the plane with the same radius. Then the collection of images
$\{j_{n(x)}(B_x)~:~x \in K\}$
is an open cover of $K$. So there is a finite subcover indexed by the
subset $\{x_1,\ldots, x_k\} \subset K$. We add the basepoint of $P_{\infty}$ to this set 
and call it $x_0$. Consider the collection
$$\sI=\{(i,j) \in \{0,\ldots,k\}^2~:~j_{n(x_i)}(B_{x_i}) \cap j_{n(x_j)}(B_{x_j}) \neq \emptyset\}.$$
Then for each $(i,j) \in \sI$, we can choose points $y \in B_{x_i}$ and $z \in B_{x_j}$
so that $j_{n(x_i)}(y)=j_{n(x_j)}(z)$. By Corollary \ref{cor:fusion finiteness}, there is a finite subset
$\sF \subset \{P_n\}$ containing $P_{n(x_i)}$ and $P_{n(x_j)}$ so that the immersions $P_{n(x_i)} \imm \Fuse \sF$
and $P_{n(x_j)} \imm \Fuse \sF$ send $y$ and $z$ to the same point. Because we are working with a directed sequence, we just
have $\Fuse \sF=P_{N(i,j)}$ where $N(i,j)$ is the maximal index of a planar surface in $\sF$. So, there is an $N=N(i,j)$ so that 
$\iota_{m(x_i),N}(y)=\iota_{m(x_j),N}(z)$. Then, because $j_{n(x_i)}(B_{x_i}) \cap j_{n(x_j)}(B_{x_j})$ is path connected, 
the map $j_{N(i,j)}$ restricted to 
$$\iota_{n(x_i),N(i,j)}(B_{x_i}) \cup \iota_{n(x_j),N(i,j)}(B_{x_j})$$
is injective. We have defined $N(i,j)$ for all $(i,j) \in \sI$. Let $M=\max_{(i,j) \in \sI} N(i,j)$. Then $j_M$ restricted to 
$$\bigcup_{i=0}^k \iota_{n(x_i),M}(B_{x_i})$$
is injective. Then, because $j_M$ is a a local homeomorphism, this restriction is a
homeomorphism onto its image. The image contains $K$. We conclude
that we can set $\tilde K$ equal to the preimage of $K$ under this restriction of $j_M$. 
This verifies the existence of $\tilde K$ and proves that the first statement of Proposition \ref{prop:conv} holds.

Now we consider the second statement of Proposition \ref{prop:conv}.
Suppose that $Q$ is a planar surface, and there is an immersion $k:Q \imm P_n$
for some $n$. Then, $j_n \circ k:Q \imm P_\infty$. This proves the second statement
needed from Proposition \ref{prop:conv}, and concludes the proof that $P_n \to P_\infty$. 
\end{proof}

\subsection{Proofs}

\begin{proof}[Proof of Theorem \ref{thm:compactness}]
Let $P$ be a planar surface. We will prove $\tM \smallsetminus \tM_{\not \imm}(P)$ is compact.
Since $\tM$ is second-countable, so is the subspace $\tM \smallsetminus \tM_{\not \imm}(P)$.
In the presence of second-countability, compactness is implied by {\em sequentially
compactness}, i.e., that every sequence in $\tM \smallsetminus \tM_{\not \imm}(P)$ has a convergent subsequence \cite[Chapter 11, Theorem 1.10]{Joshi}.
So, let $\langle Q_n \rangle_{n \geq 0}$ be a sequence in $\tM \smallsetminus \tM_{\not \imm}(P)$.
We will provide an algorithm which produces a convergent subsequence $\langle Q_{n_k} \rangle_{k \geq 0}$ converging to some limit $R \in \tM \smallsetminus \tM_{\not \imm}(P)$.

Recall that there are only countably many open rational rectangular unions which are homeomorphic to open disks.
See Corollary \ref{cor:countably many classes}. Let $\langle P_m \in \tM \rangle_{m \geq 1}$
be a sequence which enumerates all of these rectangular unions. 
We will construct a subsequence $\langle P_{m_k}\rangle$ of $\langle P_m \rangle$
while simultaneously producing $\langle Q_{n_k}\rangle$. 

Our algorithm is really an inductive sequence of definitions:
\begin{enumerate}
\item[(1)] Set $R_0=P$.
\item[(2)] Set $m_0$, $n_0=0$ and $k=1$.
\item[(3)] Set $\sI_0=\{n~:~n \geq 1\}$. 
\item[(4)] For each successive integer $m \geq 1$, if $P_m \imm Q_n$ for infinitely many $n \in \sI_{k}$, then perform the following steps:
\begin{enumerate}
\item[(a)] Set $m_k=m$. 
\item[(b)] Set $R_{k}=R_{k-1} \fuse P_{m_k}$.
\item[(c)] Set $n_k = \min \{n \in \sI_{k-1}~:~P_{m_k} \imm Q_n\}$. 
\item[(d)] Set $\sI_{k}=\{n \in \sI_{k-1}~:~\text{$P_{m_k} \imm Q_n$ and $n > n_k$}\}$.
\item[(e)] Increment $k$. (Reassign $k$ to be $k+1$.)
\end{enumerate}
\end{enumerate}
Observe that by definition of the fusion, $R_{k-1} \imm R_k$ for all $k \geq 1$.
So by taking a direct limit, we can define $R = \lim_{k \to \infty} R_k$; see Proposition \ref{prop:direct limit}. We make several further remarks about this construction:
\begin{enumerate}
\item[(R1)] For each $k$ and each $l \geq k$, $P_{m_k} \imm Q_{n_l}$. ({\em Proof:} This holds when $k=l$
by definition of $m_k$ and $n_k$. It holds when $l>k$, because each such $n_l$ lies in $\sI_k$.)
\item[(R2)] We have $R_k=P \fuse P_{m_1} \fuse \ldots \fuse P_{m_k}$. 
\item[(R3)] For each $k$ and each $l \geq k$, $R_k \imm Q_{n_l}$. ({\em Proof:} By (R1), each $P_{m_j} \imm Q_{n_l}$ for $j \leq k$. So by the Fusion Theorem and (R2), $R_k \imm Q_{n_l}$.) 
\end{enumerate}

We claim that the subsequence $\langle Q_{n_l}\rangle$ also converges to $R$. To prove this, we will use the convergence criterion of Proposition \ref{prop:conv}. First suppose that $K \in \cdisk(R)$. We will
prove that $K \imm Q_{n_l}$ for $l$ sufficiently large. Since $\langle R_k \rangle$ converges to $R$,
there is a $L$ so that $K \imm R_l$ for $l>L$. So by composing these immersions with the immersions given by remark (R3), we see $K \imm Q_{n_l}$ for $l>L$. 

Now let $U$ be a planar surface. We will show that if $U$ immerses in infinitely many $Q_{n_k}$, then $U \imm R$. 
By Proposition \ref{prop:supremum}, it suffices to prove that every compact disk $K \in \cdisk(U)$
immerses in $R$. Fix $K \in \cdisk(U)$. By Theorem \ref{thm:nested}, there
is an open rational rectangular union in $\disk(U)$ which contains $K$. By definition of $P_m$,
there is an $m$ so that $P_m$ is isomorphic to this union. In particular $P_m \imm U$. 
Since $U \imm Q_{n_k}$ for infinitely many $k$,
it must be true that $P_m$ immerses in infinitely many $Q_{n_k}$. It follows that
$P_m=P_{m_k}$ for some $k$. Then, by the definition of the fusion,
$P_{m_k} \imm R_k$. Because $R_k \imm R$, we know $P_{m_k} \imm R$. Finally because $K \imm P_{m_k}$, 
we know that $K \imm R$. 
\end{proof}

\begin{proof}[Proof of Corollary \ref{cor:local compactness}]
To see $\tM$ is locally compact, choose $P \in \tM$. Let $K$ be a closed disk in $P$. Then 
$$U=\{Q \in \tM:~K^\circ \imm Q\}$$
is a compact and contains the neighborhood $\tM_\imm(K)$ of $P$.

Since $\tM$ is locally compact, so is $\tM \times \R^2$.
We will show that the map $\tE \to \tM \times \R^2$ defined by $(P,p) \mapsto \big(P,\dev(p)\big)$ is a local homeomorphism. Local compactness of $\tE$ follows because we can pullback compact neighborhoods. The map is clearly continuous since both $\tpi$ and $\dev$ are continuous. It is locally continuously invertible because any $(P,p)$ lies in a set of the form $\tE_\imm(K,U)$ where
$U$ is so small that $\dev|_U$ is a homeomorphism onto its image. So, we can recover a point $(Q,q) \in \tE_\imm(K,U)$ from its image $\big(Q,\dev(q)\big)$
by selecting $u \in U$ to be the unique point so that $\dev(u)=\dev(q)$.
Then there is an immersion $\iota:K \imm Q$ and $\iota(u)=q$. Continuity of this local inverse is then provided by Theorem \ref{thm:Continuity of immersions}.
\end{proof}

\begin{proof}[Proof of Corollary \ref{cor:metrizability}]
Since $\tM$ and $\tE$ are locally compact and Hausdorff, they are regular \cite[Exercise 32.3]{Munkres}. Since they are second-countable and regular,
they are metrizable by Urysohn's Metrization Theorem \cite[Theorem 34.1]{Munkres}.
\end{proof}

\appendix
\section{Comparison to McMullen's Geometric Topology}
\label{sect:mcmullen}
McMullen has established a geometric topology on the space of all Riemann surfaces equipped with a base-frame paired with a holomorphic quadratic differential where the Riemann surface is allowed to be of arbitrary topological type \cite[Appendix]{McMullenAmenability}. In brief, McMullen first provides a geometric topology on the space of all Riemann surfaces with a base-frame (a choice of a tangent vector at some basepoint) by identifying such surfaces with a quotient of a simply connected rotationally symmetric domain (disk, plane or sphere) in the Riemann sphere by a subgroup of $\PSL(2,\C)$. This identification is done via the unique uniformization carrying a lift $v$ to the universal cover to the tangent vector $1$ at the origin.
Then given a triple $(X,v,q)$ where $(X,v)$ is a Riemann surface with a base-frame and $q$ is a holomorphic quadratic differential, McMullen considers the lift of $q$ to a holomorphic quadratic differential $\tilde q$ on the universal cover as identified with the domain mentioned above. Such differentials then correspond to
the choice of a holomorphic function on the domain satisfying $\tilde q=\phi(z) dz^2$.
The space of pairs $(X,q)$ is then topologized in such a way so that sequences converge if the underlying Riemann surfaces converge and the functions $\phi$ constructed as above converge uniformly on compact sets. 
See \cite[Appendix]{McMullenAmenability} for a detailed description of the geometric topology.

McMullen's topology differs from the topology here in a philosophical way. The immersive topology introduced here is intended to be a ``geometric topology for the flat structure'' and was designed to impose the minimal amount of structure necessary for making certain elementary geometric arguments (through the use of immersions and embeddings). In contrast, McMullen's geometric topology makes convergence of underlying Riemann surfaces necessary for convergence of flat structures. 

We will demonstrate in this section that our approach gives a less strict notion of convergence. However, it will follow from work in subsequent papers on the immersive topology that this topology agrees with McMullen's geometric topology on natural subspaces such as on the closures of strata of finite genus translation surfaces.

Observe there is a differences in the spaces being topologized by the geometric topology and the immersive topology. The geometric topology works with quadratic differentials with zeros, while our topology only applies to surfaces with a translation structure and no singularities. So, we hope to compare the immersive topology on the space of all translation surfaces with the geometric topology applied to an appropriate subset.
Unfortunately, our planar surfaces involve only the choice of a basepoint, so there is freedom in the choice of a triple $(X,v,q)$ representing a planar surface $P$ corresponding to the choice of a non-zero vector in the tangent plane at the basepoint. This is a minor point, but must be dealt with to make a precise statement relating the topologies. We show:

\begin{theorem}
\label{thm:different}
There is a sequence of planar surfaces $P_n$ which is convergent in the immersive topology so that for any sequence
$(X_n,v_n,q_n)$ translation equivalent to $P_n$, no subsequence of $(X_n,v_n,q_n)$ converges to a representation of a flat surface in the geometric topology.
\end{theorem}

\begin{remark}
\label{rem:compactness difference}
As a consequence of this theorem, in the geometric topology the space of flat surfaces with a lower bound on the injectivity radius is not compact. This property holds for the immersive topology by Theorem \ref{thm:compactness}.
\end{remark}


Our construction utilizes the following elementary observation in complex analysis:

\begin{proposition}
Suppose $P$ is a planar surface uniformized by the plane, and $p \in P$ is a point distinct from the basepoint.
Let $\tilde P$ be the universal cover of $P \smallsetminus p$. Then $\tilde P$ is also uniformized by the plane.
\end{proposition}

Essentially the uniformizing map of $\tilde P$ can be expressed as the uniformizing map for $P$ precomposed with the exponential map. We will not provide a formal proof.

\begin{proof}[Proof of Theorem \ref{thm:different}]
First we produce the sequence $P_n$ inductively. Choose a countable dense subset $\{z_i:i=0,1,2,\ldots\}$
of the unit circle. For each $n \geq 0$, let $D_n$ be the closed unit disk with $\{z_i:~0 \leq i < n\}$ removed.
Define $D_{-1}$ to be the closed unit disk.
Let $P_0$ be the plane viewed as a planar surface. Observe that there is an embedding of the closed unit disk, $\epsilon_0:D_{-1} \emb P_0$. Now we will begin our induction. Assume that $P_n$ is uniformized by the plane,
and there is an embedding $\epsilon_n:D_{n-1} \emb P_n$. Define $p_n = \epsilon_n(z_{n})$. Then let $P_{n+1}$ be the universal cover of $P_n \smallsetminus \{p_n\}$. By the proposition above, $P_{n+1}$ is also uniformized by the plane. By restriction we see that $D_n \emb P_n$. Since $D_n$ is simply connected this embedding lifts to an embedding
$\epsilon_{n+1}:D_n \emb P_{n+1}$. This defines the sequence $P_n$.

We will now verify convergence in the immersive topology. Let $E_i$ be the union of the open unit disk and the singleton $\{z_i\}$. Observe that by construction $E_i \not \emb P_i$.
In fact, $E_i \not \emb P_n$ for $n \geq i$. This may be seen by induction; each $P_{n+1}$ immerses in $P_n$
by construction since $P_{n+1}$ is a covering of $P_n$ with a point removed. So, if there was
an embedding of $E_i$ into $P_n$ for $n \geq i$, then composing with a finite list of immersions would give
and immersion of $E_i$ into $P_i$, but such an immersion must be an embedding which is a contradiction.
It then follows that as planar surfaces $P_n$ converges to the unit disk in the immersive topology,
since no open disk larger than the open unit disk can be immersed in infinitely many $P_n$; compare Example \ref{ex: convergence}.

Now we will consider the behavior of the sequence in the geometric topology. 
Select for each $P_n$ a uniformizing map $h_n:\C \to P_n$ carrying $0$ to the basepoint of $P_n$. (This involves a choice for each $n$.) Let $f_n:\C \to \C$ be $\dev \circ h_n$. Let $q_n=f'_n(z)^2 dz^2$. 
Suppose after passing to a subsequence, $(\C,v,q_{n_k})$ converges (where $v$ is the tangent vector $1$ based at zero). The limit is then of the form $(\C,v,q)$ for some quadratic differential $q$ on $\C$ and we know by definition of the topology that $q=\phi(z) dz^2$ and $f_{n_k}'(z)^2$ converges to $\phi$ uniformly on compact sets. By passing to another subsequence (because of the square), we can assume that $f_{n_k}(z)$ converges to a holomorphic function $f:\C \to \C$ satisfying $f(0)=0$ uniformly on compact sets. 

Now assume that the limiting differential represents a surface. Then in particular, $f$ is not identically zero,
and so $f$ misses at most one point of the plane. In particular, there is a $z \in \C$ so that $|f(z)|>1$. 
Then $|f_{n_k}(z)|>1$ for $k$ sufficiently large. Consider the embedding $\epsilon_{n_k}: D_{n_k-1} \emb P_n$ and recall that $\epsilon_{n_k}$ can not be extended to a larger subset of the closed disk. 
Let $\Gamma_{k}$ be the intersection of $D_{n_k-1}$ with the unit circle, which consists of $n_k-1$ arcs of the unit circle running between the points $\{z_i:0 \leq i< n_{k}-1\}$. Let 
$C_{k}=f_{n_k}^{-1} \circ \epsilon_{n_k}(\Gamma_{k})$, which consists of $n_k-1$ disjoint arcs in $\C$ which extend to closed arcs in the Riemann sphere $\hat \C$ initiating and terminating at $\infty$. For $k$ sufficiently large, one of these arcs say $c_{k} \subset C_{k}$ separates $0$ from $z$ since $|f_{n_k}(z)|>1$. Choose a large closed ball about the origin $B \subset \C$ so that $z$ lies in the interior of $B$. Then by passing to a subsequence, we can assume that the intersections $c_{k} \cap B$ converge
in the Hausdorff topology on closed subsets of $B$ to some closed subset $c \subset B$. Each of $c_k \cap B$ has some definite length since $c_k$ separates $0$ from $z$. The length of the associated arc $\gamma_k \subset \Gamma_k$ of $D_{n_k-1}$ can be computed in terms of $f_{n_k}'$ as
\begin{equation}
\label{eq:integral}
\int_{c_k} |f_{n_k}'(z)| dz.
\end{equation}
Because of our choice of points $\{z_i\}$ dense in the circle, the length of $\gamma_k$ tends to zero,
so these integrals tend to zero. Let $w \in c \subset B$. If $|f(w)|>0$, then there would be an $\epsilon>0$ and a definite neighborhood $W$ about $w$ so that $|f_{n_k}(w')|>\epsilon$ for $w' \in W$ and $k$ large enough, but this contradicts the 
convergence of \eqref{eq:integral} to zero. Therefore, $f$ is identically zero on $c$, but this violates the fact that non-zero holomorphic functions have isolated zeros.
\end{proof}

\vspace{1em}

\noindent
{\bf Acknowledgments:} The author would like to thank Joshua Bowman for helpful conversations at the beginning of this work, Anja Randecker for numerous helpful comments, and the anonymous referee for making numerous helpful suggestions. The author thanks Matt Bainbridge for pointing out McMullen's geometric topology. The impetus for writing this article came from the conference ``International conference and workshop on surfaces of infinite type'' held at Centro de Ciencias Matem\'aticas de la UNAM in Morelia, Mexico. Early versions of this paper were written while visiting the Institute for Computational and Experimental Research in Mathematics (ICERM).

\bibliographystyle{amsalpha}
\bibliography{../../bibliography}
\end{document}